\documentclass[final]{siamltex}
\usepackage{xcolor}
\usepackage{amsmath}
\usepackage{amssymb}
\usepackage{bbm}
\usepackage{mathrsfs}
\usepackage{graphicx}
\usepackage{tikz}
\usepackage{bm}
\usepackage{ucs}
\usepackage{wrapfig}
\usepackage{color}
\usepackage{float}
\usepackage{fullpage}
\usepackage{epstopdf}
\usepackage[breaklinks,colorlinks=true,linkcolor=blue,citecolor=red,
  backref=page]{hyperref}

\newtheorem{remark}{Remark}

\newtheorem{notation}[theorem]{Notation}

\newcommand{\ind}{\,\mbox{d}}

\usepackage[normalem]{ulem}
\definecolor{rot1}{rgb}{1.000,0.000,0.000}
\definecolor{rot2}{rgb}{0.000,0.000,1.000}

\newcommand{\R}{{\mathbb R}}

\newcommand{\C}{{\mathbb C}}

\newcommand{\be}{\begin{eqnarray}}
\newcommand{\ben}{\begin{eqnarray*}}
\newcommand{\en}{\end{eqnarray}}
\newcommand{\enn}{\end{eqnarray*}}

\newcommand{\pa}{\partial}
\usepackage[T1]{fontenc}
\usepackage{subfigure}

\begin{document}

\title{Modified sampling method with near field measurements}
\author{Xiaodong Liu\footnotemark[1], Shixu Meng\footnotemark[2], and Bo Zhang\footnotemark[3]}
\renewcommand{\thefootnote}{\fnsymbol{footnote}}
\footnotetext[1]{Academy of Mathematics and Systems Science, Chinese Academy of Sciences, Beijing 100190, China. {\tt xdliu@amt.ac.cn}}
\footnotetext[2]{Academy of Mathematics and Systems Science, Chinese Academy of Sciences,
Beijing 100190, China.  {\tt shixumeng@amss.ac.cn}}
\footnotetext[3]{Academy of Mathematics and Systems Science, Chinese Academy of Sciences,
Beijing 100190, China and School of Mathematical Sciences, University of Chinese Academy of Sciences,
Beijing 100049, China.  {\tt b.zhang@amt.ac.cn}}
\maketitle

\begin{abstract}
This paper investigates the inverse scattering problems using sampling methods with near field measurements. The near field measurements appear in two classical inverse scattering problems: the inverse scattering for obstacles and the interior inverse scattering for cavities. We propose modified sampling methods to treat these two classical problems using near field measurements without making any asymptotic assumptions on the distance between the measurement surface and the scatterers. We provide theoretical justifications based on the factorization of the near field operator in both symmetric factorization case and non-symmetric factorization case. Furthermore, we introduce a data completion algorithm  which allows us to apply the modified  sampling methods to treat the limited-aperture inverse scattering problems. Finally numerical examples are provided to illustrate the modified  sampling methods with both full- and limited- aperture near field measurements.
\end{abstract}

\section{Introduction}  \label{Introduction}
Inverse scattering  plays an important role in non-destructive testing, medical imaging, geophysical exploration and numerous problems associated with target identification. In the last thirty years, sampling methods for shape reconstruction in inverse scattering problems have attracted a lot of interest.
Classical examples include the linear sampling method  \cite{ColtonKirsch}, the singular sources method  \cite{potthast2010study}
and the factorization method  \cite{Kirsch98}.
The basic idea is to design an indicator which is large inside the underlying scatterer and relatively small outside.
We refer to the monographs  \cite{CaCo,CK,kirsch2008factorization} for a comprehensive introduction.
Recently, other types of sampling methods are proposed for the inverse scattering problems,
e.g., orthogonality sampling   \cite{potthast2010study,griesmaier2011multi,harris2020orthogonality}, direct sampling method  \cite{ItoJinZou,LiuIP17}, single-shot method   \cite{LiLiuZou},
reverse time migration \cite{CCH2013}.
These sampling methods inherit many advantages of the classical ones, e.g., they are independent of any a priori information on the geometry and physical properties of the unknown objects. The main feature of these  sampling methods is that only inner product of the measurements with some suitably chosen functions is involved in the imaging function and thus  these  sampling methods are robust to noises.
In all of the sampling methods, the measurements may be divided into two types: far-field measurements and near field measurements. Generally speaking, the analysis is more rich in the far-field case compared to the near field case. This paper is devoted to  sampling methods in the near field case with both full- and partial- aperture data.

Compared to orthogonal/direct sampling methods using  far-field measurements \cite{potthast2010study,griesmaier2011multi,LiuIP17,harris2020orthogonality}, the study on near field orthogonal/direct sampling method is relatively limited. The near field measurements appear in two classical inverse scattering problems:  the  inverse scattering for obstacles (see for instance the monograph \cite{CK}) and the interior inverse scattering for cavity \cite{QCo2,QCa,L,CaCoMe,MHC}. Is it possible to treat these two classical inverse scattering problems using one unified framework? The answer is yes and this is one of our main results in this paper.  We survey literatures on both  problems as follows. Concerning the inverse scattering for obstacles, there have been considerable works on the reverse time migration \cite{CCH2013} and  direct sampling methods \cite{ItoJinZou} using  near field measurements. In these works, numerical algorithms illustrated the performance of the imaging method and the analysis was done when the measurement surface is not close to the scatterer. Our work differs from these existing literatures because we propose modified sampling methods to treat these two classical problems using near field measurements without making any asymptotic assumptions on the distance between the measurement surface and the scatterers.
%
One difficulty in this treatment is the non-symmetric factorization \cite{Hu2014,MHC} of the near field operator and the other difficulty is the necessity to consider the ``distance'' related fundamental solution (as contrary to considering the plane wave in the far-field case). We show in this paper how to design a sampling method to overcome these difficulties. Furthermore, the sampling method is completely theoretically justified in the sense that the proposed imaging function has both an upper and lower bound which peak when the sampling point is at the obstacle boundary. On the other aspect, there seems to be no attempts to design orthogonal/direct imaging methods in the interior inverse scattering for cavity where the measurements are in the near field naturally. In the cavity setting, it is not possible to asymptotically consider the near field measurements  whereby the existing methods \cite{CCH2013,ItoJinZou} do not apply.
Fortunately, our proposed sampling method is in a unified framework, which allows us to treat the interior inverse scattering for cavity by a similar sampling method using near field measurements. Once again, this modified  sampling method provides both theoretical justifications and efficient numerical algorithms. This seems a first attempt on orthogonal/direct sampling methods in the interior inverse scattering for cavity.

In many cases of practical interest, it is difficult or even impossible to obtain the {\em full-aperture} measurements, this motivates us to consider our modified sampling methods with {\em limited-aperture} measurements.  Reconstruction algorithms have been developed using the {limited-aperture} data directly
\cite{BaoLiu,ColtonMonk06,IkehataNiemiSiltanen,kirsch2008factorization,LuXuXu2012AA,MagerBleistein1978, Zinn1989}. Alternatively,  \cite{LiuSun19, LuXuXu2012AA} first recover the full-aperture data and then solve the inverse problems. Recently two novel data completion algorithms were proposed \cite{DLMZ2021} for the inverse scattering problems in the far-field case, see also \cite{BORCEA2019556} for the waveguide case. We aim to apply a similar data completion algorithm in the near field case to recover the full-aperture data  and then apply our modified sampling methods.

The paper is further organized as follows. In Section \ref{section model}, we give the mathematical formulation of the two classical inverse scattering problems with near field measurements:   the inverse scattering for obstacles and the interior inverse scattering for cavities. We also provide a preliminary result on the coercivity estimate of the single-layer operator. The modified  sampling methods using near field measurements are investigated  in Section \ref{section obstacle} for obstacles  and in Section \ref{section cavity} for cavities, respectively. Such modified sampling methods give both theoretical justifications and numerical algorithms, no matter the associated near field operators have  symmetric factorizations or not. The key idea is to design properly chosen functions to overcome the above two difficulties associated with non-symmetric factorization and ``distance'' related fundamental solution.   Section \ref{section limited-aperture} is devoted to a data completion algorithm which recovers the full-aperture data and allows us to apply the modified sampling methods in the previous sections. Finally, numerical examples are provided in Section \ref{section numerical examples} to illustrate the modified sampling methods with both full- and limited- aperture near field measurements.

\section{Mathematical model and setup} \label{section model}
\subsection{Mathematical models}
We consider two classical inverse scattering problems with near field measurements:  the inverse scattering for obstacles and the interior inverse scattering for cavities. Throughout the paper we focus on the two dimensional case. The three dimensional case is similar yet to be done.  In both problems, let $k>0$ be the  wave number. A point source $\phi(\cdot;y)$ at $y$ is the fundamental solution with the following explicit expression
\begin{equation} \label{def fundamental solution}
\phi(x;y)=\frac{i}{4} H^1_0(k|x-y|),\quad x\not=y,
\end{equation}
where $H^1_0$ is the Hankel function of the first kind of order zero \cite{CK}.

\vspace{.5\baselineskip}

\noindent \textbf{Inverse scattering for obstacles}: Let $\Omega \subset \mathbb{R}^2$ be an open and bounded domain with Lipschitz boundary $\partial \Omega$ such that $\mathbb{R}^2 \backslash \overline{\Omega}$  is connected. The domain $\Omega$ is referred as the obstacles. The scattering for the obstacles due to a point source $\phi(\cdot;y)$ is to find scattered wave field $u^s(\cdot;y)$ such that
\begin{eqnarray}
\Delta_x u^s(\cdot;y) + k^2 u^s(\cdot;y) = 0 \quad &\mbox{in}& \quad \mathbb{R}^2\backslash \overline{\Omega}, \quad  \label{obstacle us eqn1}\\
u^s(\cdot;y) = -\phi(\cdot;y) \quad &\mbox{on}& \quad \partial \Omega,   \label{obstacle us eqn2} \\
\lim_{r:=|x|\to \infty} \sqrt{r}  \left( \frac{\partial u^s(\cdot;y)}{\partial r} -ik u^s(\cdot;y)\right) =0.
\end{eqnarray}
This scattering problem is well-posed, see for instance \cite{CK}.
Here we have assumed a sound-soft obstacle. 

Let $\partial B:=\{x: |x|=r_o, \,r_o>0\}$ be the measurement surface and $B:=\{x: |x|<r_o, \,r_o>0\}$ includes $\Omega$ as its interior. The inverse problem is to determine $\partial \Omega$ from the following \textit{near field} measurements:
\begin{equation} \label{obstacle nfm}
\{ u^s(x;y): x, y\in \partial B\}.
\end{equation}
The inverse problem has a unique solution \cite{CK}.

\vspace{.5\baselineskip}

    \begin{figure}[ht!]
    \centering
\includegraphics[width=0.59\linewidth]{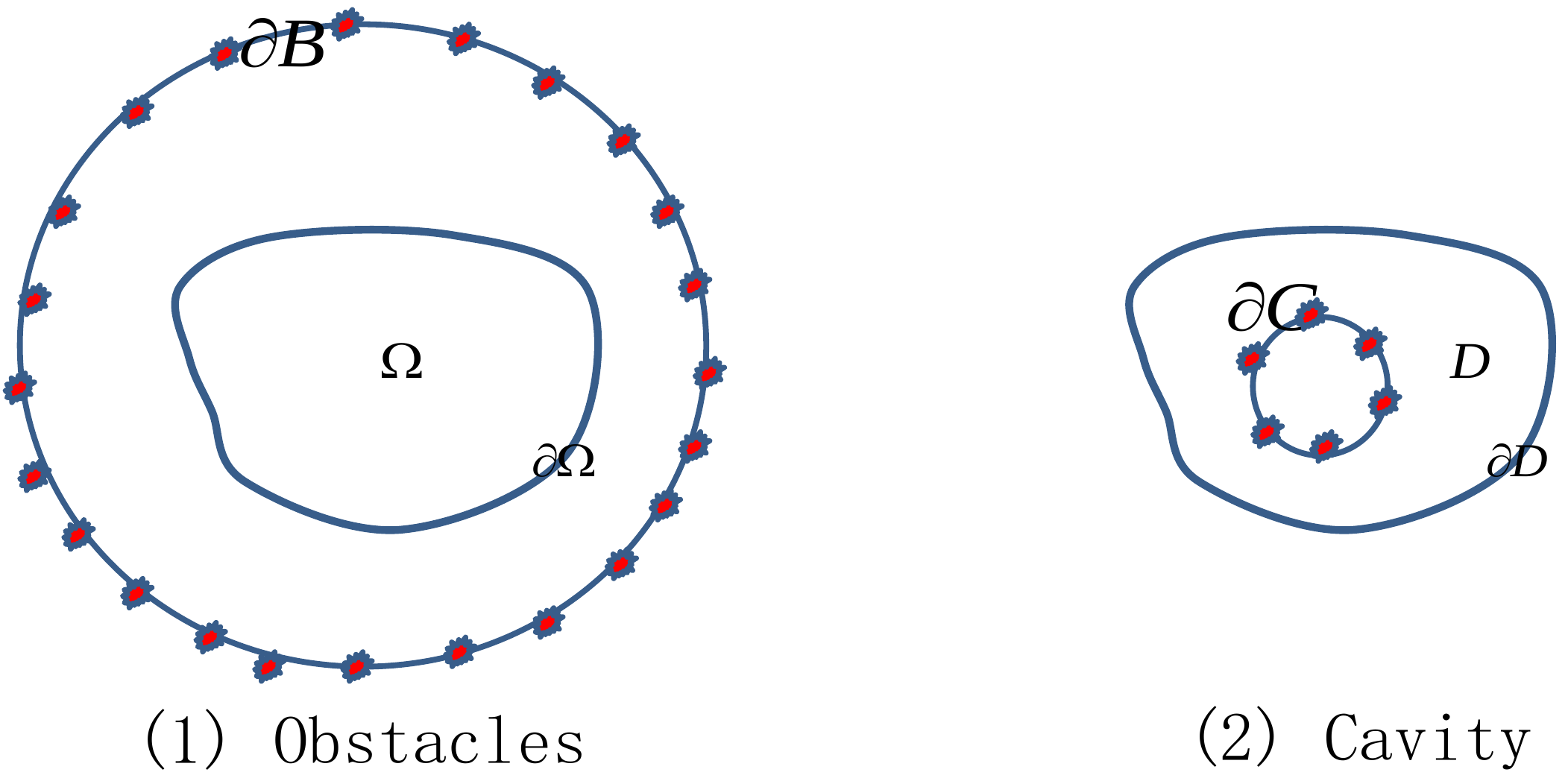}
     \caption{
     \linespread{1}
Example of the inverse scattering for obstacles and the interior inverse scattering for a cavity.
     } \label{fig geometry}
    \end{figure}

\noindent \textbf{Interior inverse scattering for cavities}:  Let $D\subset \mathbb{R}^2$ be an open, connected and bounded domain with Lipschitz boundary $\partial D$. The domain $D$ is referred as the cavity. The scattering for the cavity due to a point source $\phi(\cdot;y)$ is to find scattered wave field $u^s(\cdot;y)$ such that
\begin{eqnarray}
\Delta_x u^s(\cdot;y) + k^2 u^s(\cdot;y) = 0 \quad &\mbox{in}& \quad D, \quad  \label{cavity us eqn1}\\
u^s(\cdot;y) = -\phi(\cdot;y) \quad &\mbox{on}& \quad \partial D   \label{cavity us eqn2}.
\end{eqnarray}
If $k^2$ is not an eigenvalue of $-\Delta$ in the domain $D$, this scattering problem is well-posed, see for instance \cite{L,QCo2}. Again here we have assumed a sound-soft cavity.

Let $\partial C:=\{x: |x|=r_i, \,r_i>0\}$ be the measurement surface and $C:=\{x: |x|<r_i, \,r_i>0\}$ is in the interior of $D$. The inverse problem is to determine $\partial D$ from the following \textit{near field} measurements:
\begin{equation} \label{cavity nfm}
\{ u^s(x;y): x,y \in \partial C\}.
\end{equation}


In the interior inverse scattering for cavity problem, we always make the assumptions that $k^2$ is not an eigenvalue of $-\Delta$ in the domain $D$ and $C$. Note that  the assumption that $k^2$ is not an eigenvalue of $-\Delta$ in $C$ is not a restriction, since one can always choose $C$ such that this assumption holds. With these assumptions, the inverse problem has a unique solution \cite{L,QCo2}.

%
%
\subsection{Estimate of Single-layer operator}
%
Let $\Gamma$ be either $\partial D$ or $\partial \Omega$ and assume that $k^2$ is not an eigenvalue of $-\Delta$ in the domain  $\Omega$ and $D$, respectively. We define $ \mathrm{T}_\Gamma: H^{1/2}(\Gamma) \to H^{-1/2}(\Gamma)$ by
\begin{equation} \label{def T 1}
\mathrm{T}_\Gamma g=h,
\end{equation}
where, for any $g \in H^{1/2}(\Gamma)$, $h \in H^{-1/2}(\Gamma)$ is the unique solution to
\begin{equation} \label{def T 2}
g(x) = \int_{\Gamma} \overline{ \phi(x;y)} h(y) \ind s_y, \quad  x \in   \Gamma.
\end{equation}
We note that the above integral equation is uniquely solvable due to that $k^2$ is not an eigenvalue of $-\Delta$ in the domain $D$ or $\Omega$, see for instance \cite{McLean2000,CK,kirsch2008factorization}.

The following coercivity estimate plays an important role in the analysis of our imaging function.
\begin{lemma} \label{mo coercivity}
Let $\Gamma$ be either $\partial D$ or $\partial \Omega$.
For any $g \in H^{1/2}(\Gamma)$,
\begin{equation}
\big| \langle  \mathrm{T}_\Gamma g,g \rangle_\Gamma \big| \ge c_0 \|g\|^2_{H^{1/2}(\Gamma)}.
\end{equation}
Here $\langle \cdot, \cdot \rangle_\Gamma$ denotes the duality paring between $H^{-1/2}(\Gamma)$ and $H^{1/2}(\Gamma)$.
\end{lemma}
\begin{proof}
We prove the theorem for $\Gamma=\partial \Omega$. The proof of the case that $\Gamma=\partial D$ is exactly the same.

(a) We first show \begin{eqnarray} \label{mo coercivity proof eqn1}
-\Im \langle  \mathrm{T}_\Gamma g,g \rangle_\Gamma   <0, \quad \forall g\not=0.
\end{eqnarray}

Let $\mathrm{T}_\Gamma g=h$. Let $w^\pm(x):=\int_{\partial \Omega} \overline{ \phi(x;y)} h(y) \ind s_y$ for $x  \in  \Omega^\pm$ where $\Omega^-=\Omega$ and $\Omega^+=\mathbb{R}^2\backslash \overline{\Omega}$. From equation \eqref{def T 2}, we have from the jump relations of the single layer potential that  $\overline{h}= \frac{\partial \overline{w}^-}{\partial \nu}|_{\partial \Omega}-\frac{\partial \overline{w}^+}{\partial \nu}|_{\partial \Omega}$ (here $\nu$ is the outward normal to $\partial \Omega$). Then we have that
\begin{eqnarray*}
&&-\Im \langle  \mathrm{T}_\Gamma g,g \rangle_\Gamma = \Im \overline{ \langle h,w^- \rangle_\Gamma }= \Im  \langle   \frac{\partial \overline{w}^-}{\partial \nu}-\frac{\partial \overline{w}^+}{\partial \nu}  ,\overline{w}^- \rangle_\Gamma \\
&=& \Im  \Big( \|\nabla \overline{w} \|^2_{L^2(B_R)} +   \|\overline{w}   \| ^2_{L^2(B_R)} - \int_{\partial B_R} \frac{\partial \overline{w}^+}{\partial \nu}  {w} \ind s \Big).  \mbox{ (Integration by parts)} \\
&=& - \Im  \int_{\partial B_R} \frac{\partial \overline{w}^+}{\partial \nu}  {w} \ind s = \Im  \int_{\partial B_R} \frac{\partial {w}^+}{\partial \nu}  \overline{w} \ind s .
\end{eqnarray*}
The proof is completed if
\begin{eqnarray} \label{mo coercivity proof eqn2}
-\Im \langle  \mathrm{T}_\Gamma g,g \rangle_\Gamma = \Im  \int_{\partial B_R} \frac{\partial {w}^+}{\partial \nu}  \overline{w} \ind s <0, \quad \forall g\not=0.
\end{eqnarray}
This is valid due to the following. Assume on the contrary that
\begin{eqnarray*}
-\Im \langle  \mathrm{T}_\Gamma g,g \rangle_\Gamma = \Im  \int_{\partial B_R} \frac{\partial {w}^+}{\partial \nu}  \overline{w} \ind s \ge 0,
\end{eqnarray*}
Since $\overline{w}$ is a radiating solution to the Helmholtz equation in $\mathbb{R}^2 \backslash \overline{B_R}$, by \cite[Theorem 2.13]{CK} $\overline{w}$ vanishes in $\mathbb{R}^2 \backslash \overline{B_R}$, and by unique continuation $\overline{w}$ vanishes in $\mathbb{R}^2 \backslash \overline{\Omega}$ and hence $w|_{\partial \Omega}=0$. Since ${w}$ satisfies the Helmholtz equation in $\Omega$, and the assumption that $-k^2$ is not an eigenvalue of $-\Delta$ in $\Omega$, we conclude that $w$ vanishes in $\Omega$. Recall that $\overline{w}$ vanishes in $\mathbb{R}^2 \backslash \overline{\Omega}$, then by the jump relation of the single layer potential, we conclude that $h$ vanishes and hence $g$ vanishes. This proves \eqref{mo coercivity proof eqn1}.

(b) Note that $\mathrm{T}_\Gamma$ is the sum of a strictly coercive self-adjoint operator and a compact operator \cite[Lemma 5.38]{CK}. This together with \cite[Lemma 7.28]{CaCo} or \cite[Lemma 5.37]{CK}, we conclude that
\begin{equation*}
\big| \langle  \mathrm{T}_\Gamma g,g \rangle_\Gamma \big| \ge c_0 \|g\|^2_{H^{1/2}(\Gamma)}.
\end{equation*}
and this completes the proof.
\end{proof}

\section{Inverse scattering for extended obstacles} \label{section obstacle}
\subsection{Near field operator and its factorization}
The near field operator $\mathrm{N}: L^2(\partial B) \to L^2(\partial B)$ is denoted by
\begin{equation}
\big( \mathrm{N} g \big) (x):= \int_{\partial B} u^s(x;y) g(y) \ind s_y, \quad  x \in \partial B.
\end{equation}
To facilitate the analysis, we define $ \mathrm{H}:L^2(\partial B) \to H^{1/2}(\partial \Omega)$ by
\begin{equation} \label{obstacle H def}
\big( \mathrm{H} g \big) (x):= \int_{\partial B} \phi(x;y) g(y) \ind s_y, \quad  x \in \partial \Omega,
\end{equation}
and thereby  $ \mathrm{H}^*:H^{-1/2}(\partial \Omega) \to L^2(\partial B) $ is given by (which can be directly verified again by the definition that $\langle \mathrm{H}^* h,g\rangle_{\partial B}=\langle h,\mathrm{H}g\rangle_{\partial \Omega}$)
\begin{equation} \label{obstacle H* def}
\big(\mathrm{H}^* h \big)(x) =  \int_{\partial \Omega} \overline{\phi(x;y)} h(y)  \ind s_y, \quad \forall h \in H^{-1/2}(\partial \Omega).
\end{equation}
For any function $g$ in a complex Banach space, we define the ``conjugate operator'' $\mathrm{R}$ by $\mathrm{R} g = \overline{g}$.

The near field operator has the following factorization.
\begin{theorem} \label{obstacle thm NO factorization}
It holds that
\begin{equation} \label{obstacle NO factorization}
\mathrm{N} = -  \mathrm{R} \mathrm{H}^* \mathrm{T}_{\partial \Omega}  \mathrm{R} \mathrm{H},
\end{equation}
where $ \mathrm{T}_{\partial \Omega}: H^{1/2}(\partial \Omega) \to H^{-1/2}(\partial \Omega)$ is defined via \eqref{def T 1}--\eqref{def T 2}.
\end{theorem}
\begin{proof}
Let $w(x):=\int_{\partial B} u^s(x;y) g(y) \ind s_y$ for $x \in \mathbb{R}^2 \backslash \overline{\Omega}$. From the definition of   $\mathrm{H}$ \eqref{obstacle H def} and superposition principle, we have that
\begin{eqnarray*}
\Delta w + k^2 w = 0 \quad &\mbox{in}& \mathbb{R}^2 \backslash \overline{\Omega}, \quad \\
w = - \mathrm{H} g \quad &\mbox{on}& \quad \partial \Omega,
\end{eqnarray*}
and $w$ satisfies the Sommerfeld radiation condition. Thereby we look for
$$
w(x) = \int_{\partial \Omega}  {\phi(x;y)} \psi(y) \ind s_y
$$
with unknown density $\psi \in H^{-1/2}(\partial \Omega)$. From the boundary condition of $w$ on $\partial \Omega$, we have that
$$
\int_{\partial \Omega}  {\phi(x;y)} \psi(y) \ind s_y = - \mathrm{H} g \mbox{ on } \partial \Omega.
$$
Taking the conjugate of the above equation yields
$$
\int_{\partial \Omega}  \overline{\phi(x;y)} \overline{\psi(y)} \ind s_y = - \overline{\mathrm{H} g} \mbox{ on } \partial \Omega.
$$
This together with the definition of $\mathrm{T}_{\partial \Omega}$ \eqref{def T 1}--\eqref{def T 2} yields that the density is solved by $\overline{\psi}=\mathrm{T}_{\partial \Omega} (- \overline{\mathrm{H} g}) = \mathrm{T}_{\partial \Omega} \mathrm{R}(-  \mathrm{H} g)$. Therefore we have that, for any $x\in \partial B$,
\begin{eqnarray*}
\big( \mathrm{N} g \big) (x) &=& w(x)=\mathrm{R} \overline{w}=\mathrm{R} \int_{\partial \Omega} \overline{\phi(x;y)} \overline{\psi(y)} \ind s_y =\mathrm{R} \int_{\partial \Omega} \overline{\phi(x;y)} \big(\mathrm{T}_{\partial \Omega}  \mathrm{R}(- \mathrm{H} g) \big) \ind s_y \\
&=& - \mathrm{R} \mathrm{H}^* \mathrm{T}_{\partial \Omega} \mathrm{R} \mathrm{H} g.
\end{eqnarray*}
This completes the proof.
\end{proof}

There are other ways to factorize the near field operator \cite{Hu2014} using near field measurements. The key point is that these factorizations are non-symmetric factorizations which is in contrast to the symmetric factorization using far-field measurements. We show in the following section how to design a sampling method even though the factorization is non-symmetric. Furthermore, the sampling method is completely theoretically justified in the sense that the proposed imaging function has both an upper and lower bound which peak when the sampling point is at the obstacle boundary.
\subsection{Imaging function}
We introduce the imaging function
\begin{equation} \label{imaging function obstacle}
I_{obstacle}(z):= \big| \langle   \mathrm{N} \varphi_z, \overline{\varphi_z} \rangle_{\partial B} \big|
\end{equation}
where $\langle \cdot, \cdot \rangle_{\partial B}$ denotes the $L^2(\partial B)$-inner product, and $\varphi_z$ is given by
\begin{equation} \label{obstacle varphi def}
\varphi_{z}(y): = \sum_{n=-M}^M \frac{4}{i |y| \pi(1+\delta_{0n})}  \frac{J_n(k|z|) }{H_n^{(1)}(k|y|)}\cos(n\theta_{yz}), \quad y\in \partial B
\end{equation}
with $M$ chosen to be a positive integer and $\theta_{xy}$ denoting the angle between $x$  and $y$  for any $x,y \in \mathbb{R}^2$.  Moreover we also consider
\begin{equation} \label{obstacle varphi infty def}
\varphi_{z,\infty}(y): = \sum_{n=-\infty}^\infty \frac{4}{i |y| \pi(1+\delta_{0n})} \frac{J_n(k|z|) }{H_n^{(1)}(k|y|)} \cos(n\theta_{yz}), \quad y\in \partial B,
\end{equation}
which is the infinite series version of \eqref{obstacle varphi def}.

\begin{remark}
The function
$\varphi_{z,\infty}$ \eqref{obstacle varphi infty def} is well-defined. Indeed, note that $H_n^{(1)}(k|y|)=H_n^{(1)}(kr_o)$ does not vanish, so first $\varphi_z(y)$ \eqref{obstacle varphi def} is well-defined. Letting $M \to +\infty$, we have from the asymptotic of Hankel and Bessel functions \cite[Section 3.5]{CK} that
\begin{eqnarray*}
H_n^{(1)} (k r_0) \sim \frac{2^n (n-1)!}{\pi i (k r_o)^n}, \quad J_n(k|z|) \sim  \frac{ (k |z|)^n}{2^n  n!}, \quad n \gg 1,
\end{eqnarray*}
which yields that
\begin{equation*}
 \frac{J_n(k|z|) }{H_n^{(1)}(k|y|)} \sim  \frac{ (k |z|)^n \pi i (k r_o)^n}{2^n  n! 2^n (n-1)!}
\end{equation*}
and hence by ratio test $\varphi_{z,\infty}$ \eqref{obstacle varphi infty def} is well-defined for sampling points $z$ in a bounded sampling region.
\end{remark}

The next theorem gives the explicit expressions of $\mathrm{H} \varphi_z (x)$ and $\mathrm{H} \varphi_{z,\infty} (x)$.

\begin{theorem} \label{H varphi theorem}
It holds that
\begin{equation} \label{H varphi eqn}
\Big( \mathrm{H} \varphi_z \Big) (x) =  \sum_{n=-M}^M J_n(k|x|) J_n(k|z|) \cos(n\theta_{xz}), \quad x\in \partial \Omega,
\end{equation}
and
\begin{equation}  \label{H varphi infty eqn}
\Big( \mathrm{H} \varphi_{z,\infty} \Big) (x) =  \sum_{n=-\infty}^\infty J_n(k|x|) J_n(k|z|) \cos(n\theta_{xz}) = J_0(k|x-z|), \quad x\in \partial \Omega.
\end{equation}
Here the convergence is uniformly for $x\in \partial \Omega$ and $z$ in a bounded sampling region. 
\end{theorem}
\begin{proof}
We first prove \eqref{H varphi eqn}.
From \cite{CaCo,CK},
\begin{equation*}
\phi(x;y) = \frac{i}{4}\sum_{n=-\infty}^\infty H_n^{(1)}(k|y|) J_n(k|x|) \cos(n\theta_{xy}), \quad x\in\partial \Omega, y\in \partial B.
\end{equation*}
This together with \eqref{obstacle varphi def} yields that
\begin{eqnarray*}
\Big( \mathrm{H} \varphi_z \Big) (x) &=&  \int_{\partial B} \phi(x;y) \varphi_z(y) \ind s_y \\
&=& r_o \int_{-\pi}^\pi \Big(\frac{i}{4}\sum_{n=-\infty}^\infty H_n^{(1)}(kr_o) J_n(k|x|) \cos(n\theta_{xy})  \Big) \\
&&\qquad \times \Big( \sum_{n=-M}^M\frac{4}{i r_o\pi(1+\delta_{0m})}  \frac{J_n(k|z|) }{H_n^{(1)}(kr_o)} \cos(m\theta_{yz}) \Big) \ind \theta_y \\
&=&  \sum_{n=-M}^M J_n(k|x|)   J_n(k|z|) \cos(n\theta_{xz}),  \end{eqnarray*}
where in the last step we have applied the orthogonality property  that
\begin{eqnarray}
&& \int_{-\pi}^\pi \cos(n\theta_{xy})\cos(m\theta_{yz}) \ind \theta_y \nonumber\\
 &=& \int_{-\pi}^\pi \big( \cos (n \theta_x) \cos (n \theta_y) + \sin (n \theta_x) \sin (n \theta_y) \big) \big( \cos (m \theta_z) \cos (m \theta_y) + \sin (m \theta_z) \sin (m \theta_y) \big)  \ind \theta_y \nonumber\\
 &=& \cos (n \theta_x) \cos (n \theta_z) \pi(1+\delta_{0n}) +  \sin (n \theta_x)  \sin (n \theta_z)  \pi = \pi (1+\delta_{0n})  \cos(n\theta_{xz}),\label{cos thetaxz orthogonality}
\end{eqnarray}
where $\delta$ is the Kronecker delta.
 This completes the proof of \eqref{H varphi eqn}. Letting $M\to \infty$ and using the asymptotic of the Bessel function \cite[Section 3.5]{CK}, we prove
\begin{equation*}
\Big( \mathrm{H} \varphi_{z,\infty} \Big) (x) =  \sum_{n=-\infty}^\infty J_n(k|x|) J_n(k|z|) \cos(n\theta_{xz}), \quad x\in \partial \Omega,
\end{equation*}
where the convergence is uniformly for $x\in \partial \Omega$ and $z$ in a bounded sampling region.
Note that \cite{CaCo,CK} when $|x|>|z|>0$,
\begin{equation*}
 H_0^{(1)}(k|x-z|) = \sum_{n=-\infty}^\infty H_n^{(1)}(k|x|) J_n(k|z|) \cos(n\theta_{xz}), \quad |x|>|z|>0
\end{equation*}
and when $|z|>|x|>0$
\begin{equation*}
 H_0^{(1)}(k|x-z|) = \sum_{n=-\infty}^\infty H_n^{(1)}(k|z|) J_n(k|x|) \cos(n\theta_{zx}), \quad |z|>|x|>0,
\end{equation*}
we then take the real parts of the above equations to get
$$
 \sum_{n=-\infty}^\infty J_n(k|x|) J_n(k|z|) \cos(n\theta_{xz}) =  J_0(k|x-z|), \quad x\in \partial \Omega.
$$
This proves \eqref{H varphi infty eqn}. This completes the proof.
\end{proof}

The following theorem plays an important role in the analysis of the imaging function. Recall the assumption that $k^2$ is not an eigenvalue of $-\Delta$ in the domain  $\Omega$.
\begin{theorem} \label{obstacle Iz theorem}
There exists positive constants $c'_1, c''_1$ and $c'_2, c''_2$ independent of $z$ such that
\begin{equation}
c'_1 \|  \mathrm{H} \varphi_z\|^2_{H^{1/2}(\partial \Omega)}\le \big| I_{obstacle}(z)\big| \le c'_2  \|  \mathrm{H} \varphi_z\|^2_{H^{1/2}(\partial \Omega)}.
\end{equation}
and
\begin{equation}
c''_1 \|  \mathrm{H} \varphi_{z,\infty}\|^2_{H^{1/2}(\partial \Omega)}\le \big| I_{obstacle}(z)\big| \le c''_2  \|  \mathrm{H} \varphi_{z,\infty}\|^2_{H^{1/2}(\partial \Omega)}.
\end{equation}
\end{theorem}
\begin{proof}
From the factorization of $\mathrm{N}$, we have that
\begin{equation} \label{obstacle Iz proof eqn1}
- \langle   \mathrm{N} \varphi_z, \overline{\varphi_z} \rangle_{\partial B} =\langle   \mathrm{R} \mathrm{H}^* \mathrm{T}_{\partial \Omega}  \mathrm{R} \mathrm{H} \varphi_z, \overline{\varphi_z} \rangle_{\partial B} =\overline{ \langle \mathrm{H}^* \mathrm{T}_{\partial \Omega}   \mathrm{R} \mathrm{H} \varphi_z, \varphi_z \rangle}_{\partial B} =\overline{ \langle  \mathrm{T}_{\partial \Omega}  \mathrm{R} \mathrm{H} \varphi_z, \mathrm{H}\varphi_z \rangle}_{\partial \Omega} .
\end{equation}
By Theorem \ref{H varphi theorem}, we have that $\mathrm{H}\varphi_z$ is indeed a real-valued function in $H^{1/2}(\partial \Omega)$, therefore
\begin{eqnarray*}
\big| I_{obstacle}(z)\big| =  \big| \langle  \mathrm{T}_{\partial \Omega}   \mathrm{H} \varphi_z, \mathrm{H}\varphi_z \rangle \big|.
\end{eqnarray*}
Therefore we can apply the coercivity estimate in Lemma \ref{mo coercivity} to derive that
\begin{eqnarray*}
\big| I_{obstacle}(z)\big| =   \big| \langle  \mathrm{T}_{\partial \Omega}  \mathrm{H} \varphi_z, \mathrm{H}\varphi_z \rangle \big| \ge c'_1 \|  \mathrm{H} \varphi_z\|^2_{H^{1/2}(\partial \Omega)}.
\end{eqnarray*}
The upper bound follows from the fact that $\mathrm{T}_{\partial \Omega}$ is bounded. Note that both the positive constants $c'_1$ and $c'_2$ are independent of $z$. Using exactly the same argument, the same estimate (with  constants $c''_1$ and $c''_2$) also holds when replacing $\varphi_{z}$ by $\varphi_{z,\infty}$. This completes the proof.
\end{proof}

\begin{remark}
Though the factorization  \eqref{obstacle NO factorization} of the near field operator  is non-symmetric, we are still able to design our modified sampling method. The key point is that $\mathrm{H}\varphi_z$ and $\mathrm{H}\varphi_{z,\infty}$ are indeed real-valued functions in $H^{1/2}(\partial \Omega)$ by Theorem \ref{H varphi theorem}. Furthermore, the sampling method is completely theoretically justified in the sense that the proposed imaging function has both an upper and lower bound which peak when the sampling point is at the obstacle boundary.

\end{remark}

It is seen from Theorem \ref{H varphi theorem} that
\begin{equation}
\Big( \mathrm{H} \varphi_{z,\infty} \Big) (x) = J_0(k|x-z|),
\quad x\in \partial \Omega,
\end{equation}
and when $M$ is large,
\begin{equation}
\Big( \mathrm{H} \varphi_z \Big) (x) \approx
J_0(k|x-z|),
\quad x\in \partial \Omega.
\end{equation}
In this sense, we find that when the sampling point $z$ approaches $x \in \partial \Omega$, $\big(\mathrm{H} \varphi_z \big)(x)$ peaks  when $z$ coincides with $x$.
Unfortunately, it is not possible to plot  $\mathrm{H} \varphi_z\|^2_{H^{1/2}(\partial \Omega)}$ since $\Omega$ is the unknown obstacle. What is significant about Theorem \ref{obstacle Iz theorem} is that it shows that the imaging function $I_{obstacle}(z)$ is qualitatively the same as $\|  \mathrm{H} \varphi_z\|^2_{H^{1/2}(\partial \Omega)}$ without the knowledge of $\Omega$.  
Therefore from Theorem \ref{obstacle Iz theorem} and Theorem \ref{H varphi theorem}, we  conclude that  the imaging function $I_{obstacle}(z)$ peaks  when the  sampling point $z$ is at the obstacle boundary $\partial \Omega$.

\begin{remark}
The main feature of sampling methods in the literature \cite{CCH2013,ItoJinZou} is that only inner product of the measurements with some suitably chosen functions is involved in the imaging function. In these works, numerical algorithms illustrated the performance of the imaging method and the analysis was done when the measurement surface is not close to the scatterer, which might be due to that those suitably chosen functions are ``distance'' related function (as contrary to considering the plane wave in the far-field case), e.g. depending on the distance between the measurement surface and the scatterers. The chosen functions $\varphi_{z,\infty}$  \eqref{obstacle varphi infty def} and $\varphi_{z}$  \eqref{obstacle varphi def} are different from the existing chosen functions in the literature \cite{CCH2013,ItoJinZou} and they play a very important role in the analysis of our modified sampling method. Though our chosen $\varphi_{z,\infty}$  and $\varphi_{z}$ are also ``distance'' related, they ``cancel'' such ``distance'' together with $\mathrm{H} \varphi_z (x)$ and $\mathrm{H} \varphi_{z,\infty} (x)$ without making any asymptotic assumptions on the distance between the measurement surface and the scatterers (see Theorem \ref{H varphi theorem}).
\end{remark}

\begin{remark}
In the sampling method using far-field measurements \cite{LiuIP17}, it is seen from the Jacob-Anger expansion of $e^{ik x\cdot d}$ that the imaging function $I(z)$ in \cite{LiuIP17} in fact has lower and upper bound given by
$$
c_1 \Big\| \int_{|d|=1} e^{ik x\cdot d} e^{-ik z\cdot d}  \ind s_d \Big\|^2_{H^{1/2}(\partial \Omega)} \le I(z) \le c_2 \Big\| \int_{|d|=1} e^{ik x\cdot d} e^{-ik z\cdot d}  \ind s_d \Big\|^2_{H^{1/2}(\partial \Omega)},
$$
and
$$
\Big\| \int_{|d|=1} e^{ik x\cdot d} e^{-ik z\cdot d}  \ind s_d \Big\|_{H^{1/2}(\partial \Omega)} \sim c\Big\|J_0(k|x-z|) \Big\|_{H^{1/2}(\partial \Omega)} \mbox{ for some constant } c.
$$
Together with \eqref{H varphi infty eqn} in Theorem \ref{H varphi theorem}, we find that this resolution with far-field measurements is in fact the same as the resolution using our modified sampling method with near field measurements. This theoretical result is further confirmed via numerical examples in Section \ref{section numerical examples}.
\end{remark}

Finally we summarize the imaging algorithm.

\vspace{0.5\baselineskip}

\noindent\textbf{Imaging Algorithm for obstacles.}
\
\begin{itemize}
\item Collect the  near field measurements $u^s(x;y), \,x,y \in \partial B$.
\item Select a sampling region in $\mathbb{R}^2$ with a fine mesh containing $\Omega$.
\item Compute the imaging function $I_{obstacle}(z)$ in \eqref{imaging function obstacle} with $\varphi_z$ given by \eqref{obstacle varphi def}  for all sampling points.
\item Plot the imaging function $I_{obstacle}(z)$ over the sampling region to determine $\Omega$.
\end{itemize}

\section{Interior inverse scattering for cavities} \label{section cavity}
\subsection{Near field operator and its factorization}
The near field operator $\mathrm{N}_{\partial C}: L^2(\partial C) \to L^2(\partial C)$ is denoted by
\begin{equation} \label{cavity N def}
\big( \mathrm{N}_{\partial C} g \big) (x):= \int_{\partial C} u^s(x;y) g(y) \ind s_y, \quad  x \in \partial C.
\end{equation}
To facilitate the analysis, we define $ \mathrm{S}:L^2(\partial C) \to H^{1/2}(\partial D)$ by
\begin{equation} \label{cavity S def}
\big( \mathrm{S} g \big) (x):= \int_{\partial C} \phi(x;y) g(y) \ind s_y, \quad  x \in \partial D,
\end{equation}
and thereby  $ \mathrm{S}^*:H^{-1/2}(\partial D) \to L^2(\partial C) $ is given by (which can be directly verified by the definition that $\langle \mathrm{S}^* h,g\rangle_{\partial C}=\langle h,\mathrm{S}g\rangle_{\partial D}$)
\begin{equation} \label{cavity S* def}
\big(\mathrm{S}^* h \big)(x) =  \int_{\partial D} \overline{\phi(x;y)} h(y)  \ind s_y, \quad \forall h \in H^{-1/2}(\partial D).
\end{equation}
The near field operator has the following factorization.
\begin{theorem} It holds that
\begin{equation} \label{cavity NO factorization}
\mathrm{N}_{\partial C} = -  \mathrm{S}^* \mathrm{T}_{\partial D} \mathrm{S},
\end{equation}
where $ \mathrm{T}_{\partial D}: H^{1/2}(\partial D) \to H^{-1/2}(\partial D)$ is defined via \eqref{def T 1}--\eqref{def T 2}.
\end{theorem}
\begin{proof}
The proof is almost the same as the proof of Theorem \ref{obstacle thm NO factorization}. For completeness we outline the proof.
Let $w(x):=\int_{\partial C} u^s(x;y) g(y) \ind s_y$ for $x \in D$. From the definition of   $\mathrm{S}$ \eqref{cavity S def} and superposition principle, we have that
\begin{eqnarray*}
\Delta w + k^2 w = 0 \quad &\mbox{in}& \quad D, \quad \\
w = - \mathrm{S} g \quad &\mbox{on}& \quad \partial D   .
\end{eqnarray*}
This is a classical boundary value problem which can be solved via integral equation method. In particular, we look for
$$
w(x) = \int_{\partial D} \overline{\phi(x;y)} \psi(y) \ind s_y
$$
with unknown density $\psi \in H^{-1/2}(\partial D)$. From the boundary condition of $w$ on $\partial D$, we have that
$$
\int_{\partial D} \overline{\phi(x;y)} \psi(y) \ind s_y = - \mathrm{S} g \mbox{ on } \partial D.
$$
This together with the definition of $\mathrm{T}_{\partial D}$ \eqref{def T 1}--\eqref{def T 2} yields that the density is solved by $\psi=\mathrm{T}_{\partial D} (- \mathrm{S} g)$. Therefore we have that, for any $x\in \partial C$,
$$
\big( \mathrm{N}_{\partial C} g \big) (x) = w(x)= \int_{\partial D} \overline{\phi(x;y)} \psi(y) \ind s_y = \int_{\partial D} \overline{\phi(x;y)} \big(\mathrm{T}_{\partial D} (- \mathrm{S} g) \big) \ind s_y = -  \mathrm{S}^* \mathrm{T}_{\partial D} \mathrm{S} g.
$$
This completes the proof.
\end{proof}
\subsection{Imaging function}
The imaging function is given by
\begin{equation}\label{imaging function cavity}
I_{cavity}(z):=\big| \langle   \mathrm{N}_{\partial C} \psi_z, \psi_z \rangle_{\partial C} \big|
\end{equation}
where $\langle \cdot, \cdot \rangle_{\partial C}$ denotes the $L^2(\partial C)$-inner product, and $\psi_z$ is given by
\begin{equation} \label{cavity psi def}
\psi_z(y): = \sum_{n=-\mathfrak{M}}^{\mathfrak{M}} \frac{4}{i|y|\pi(1+\delta_{0n})}    \frac{J_n(k|z|)}{J_n(k|y|)}\cos(n\theta_{yz}), \quad y\in \partial C,
\end{equation}
where {\footnotesize $\mathfrak{M}$} is chosen to be a positive integer.  Note that  $k^2$ is not an eigenvalue of $-\Delta$ in $C$, therefore $J_n(k|y|)=J_n(kr_i)$ never vanishes for any $n$, i.e.  $\psi_z$ \eqref{cavity psi def} is well-defined.   We now prove the resolution analysis result for the cavity case.

\begin{theorem} \label{cavity Iz theorem}
There exists positive constants $c_1$ and $c_2$ independent of $z$ such that
\begin{equation}
c_1 \|  \mathrm{S} \psi_z\|^2_{H^{1/2}(\partial D)}\le   I_{cavity}(z) \le c_2  \|  \mathrm{S} \psi_z\|^2_{H^{1/2}(\partial D)}.
\end{equation}
\end{theorem}
\begin{proof}
From the factorization of $\mathrm{N}_{\partial C}$ \eqref{cavity NO factorization}, we have that
\begin{equation} \label{cavity Iz proof eqn1}
-\langle   \mathrm{N}_{\partial C} \psi_z, \psi_z \rangle_{\partial C} =\langle   \mathrm{S}^*\mathrm{T}_{\partial D}\mathrm{S} \psi_z, \psi_z \rangle_{\partial C} = \langle  \mathrm{T}_{\partial D} \mathrm{S} \psi_z,\mathrm{S} \psi_z \rangle_{\partial D} .
\end{equation}
Therefore we can apply the coercivity estimate in Lemma \ref{mo coercivity} to derive that
\begin{eqnarray*}
I_{cavity}(z) =  \big|  \langle  \mathrm{T}_{\partial D} \mathrm{S} \psi_z,\mathrm{S} \psi_z \rangle_{\partial D} \big|   \ge c_1 \|  \mathrm{S} \psi_z\|^2_{H^{1/2}(\partial D)}.
\end{eqnarray*}
The upper bound follows from the fact that $ \mathrm{T}_{\partial D}$ is bounded. Note that both the positive constants $c_1$ and $c_2$ are independent of $z$. This completes the proof.
\end{proof}

Theorem \ref{cavity Iz theorem} states that the imaging function $I(z)$ is qualitatively the same as $\|  \mathrm{S} \psi_z\|_{H^{1/2}(\partial D)}$. The next theorem gives the explicit expression of $\mathrm{S} \psi_z (x)$.

\begin{theorem} \label{S psi theorem}
\begin{equation}
\Big( \mathrm{S} \psi_z \Big) (x) =  \sum_{n=-\mathfrak{M}}^{\mathfrak{M}} H_n^{(1)}(k|x|) J_n(k|z|) \cos(n\theta_{xz}), \quad x\in \partial D.
\end{equation}
\end{theorem}
\begin{proof}
From \cite{CaCo,CK},
\begin{equation*}
\phi(x;y) = \frac{i}{4} \sum_{n=-\infty}^\infty H_n^{(1)}(k|x|) J_n(k|y|) \cos(n\theta_{xy}), \quad x\in\partial D, y\in \partial C.
\end{equation*}
This together with \eqref{cavity psi def} yields that
\begin{eqnarray*}
\Big( \mathrm{S} \psi_z \Big) (x) &=&  \int_{\partial C} \phi(x;y) \psi_z(y) \ind s_y \\
&=& r_i \int_{-\pi}^\pi \Big(\sum_{n=-\infty}^\infty H_n^{(1)}(k|x|) J_n(kr_i) \cos(n\theta_{xy})  \Big) \\
&&\qquad \times \Big(\sum_{n=-\mathfrak{M}}^{\mathfrak{M}} \frac{1}{r_i\pi(1+\delta_{0m})}\frac{J_n(k|z|)}{J_n(kr_i)} \cos(m\theta_{yz}) \Big) \ind \theta_y \\
&=&  \sum_{n=-\mathfrak{M}}^{\mathfrak{M}} H_n^{(1)}(k|x|)   J_n(k|z|) \cos(n\theta_{xz}),  \end{eqnarray*}
where we have applied the orthogonality property  \eqref{cos thetaxz orthogonality}.
This completes the proof.
\end{proof}

When {\footnotesize $\mathfrak{M}$} is large, it can be seen from the series expansion of Hankel and Bessel function that
\begin{equation}
\Big( \mathrm{S} \psi_z \Big) (x) \approx
H_0^{(1)}(k|x-z|), \quad |x|>|z|,
\quad x\in \partial D.
\end{equation}
We find that, as the sampling point $z$ approaches $x \in \partial D$ from the interior of the cavity $D$, $\big(\mathrm{S}\psi_z\big)(x)$ peaks  when $z$ coincides with $x$.
Therefore from Theorem \ref{cavity Iz theorem} and Theorem \ref{S psi theorem}, we may conclude that  the imaging function $I_{cavity}(z)$ peaks  when  the sampling point $z$ approaches the boundary $\partial D$ from the interior of the cavity $D$.

\begin{remark}\label{Mchoise}
We have seen a difference between the obstacle case and the cavity case. In the obstacle case, we have considered the limit case \eqref{obstacle varphi infty def} by letting $M \to \infty$ in \eqref{obstacle varphi def}. This is not feasible in the cavity case since we have the following asymptotic for $|y|=r_i$
\begin{equation} \label{J over J asymptotic}
\frac{J_n(k|z|)}{J_n(k|y|)}\sim \left(\frac{|z|}{|r_i|}\right)^n, \quad n \gg 1
\end{equation}
which does not allow us to take the limit {\footnotesize $\mathfrak{M}$} $\to \infty$ in the definition of $\psi_z$ \eqref{cavity psi def} for all sampling points $z$. It is the same reason that would limit us to consider a {\footnotesize $\mathfrak{M}$} that is not too large in the numerical examples, since otherwise  $\psi_z$ \eqref{cavity psi def} would become too large which may lead to numerical instability. Since the radius $r_i$ of $C$ appears in the denominator of of the asymptotic \eqref{J over J asymptotic}, smaller $r_i$ may lead to numerical instability for noisy measurements. We further illustrate these in the numerical examples section.
\end{remark}

We end this section by summarizing the following imaging algorithm.

\vspace{0.5\baselineskip}

\noindent\textbf{Imaging Algorithm for the cavity.}
\
\begin{itemize}
\item Collect the  near field measurements $u^s(x;y), \,x,y \in \partial C$.
\item Select a sampling region in $\mathbb{R}^2$ with a fine mesh containing $D$.
\item Compute the imaging function $I_{cavity}(z)$ in \eqref{imaging function cavity} with $\psi_z$ given by \eqref{cavity psi def}  for all sampling points.
\item Plot the imaging function $I_{cavity}(z)$ over the sampling region  to determine $D$.
\end{itemize}


\section{Imaging with limited-aperture data} \label{section limited-aperture}
In this section we consider imaging with the following limited-aperture ``backscattering'' data
\begin{equation}\label{def limited data}
\{u^s(x;y): \, x,y\in \Gamma_0 \}
\end{equation}
where $\Gamma_0:=\{r(\cos \alpha,\sin\alpha): \theta \in [-\alpha,\alpha]\}$ is a subset of the measurement circle $\Gamma$ (which is either $\partial C$ or $\partial B$), here $0<\alpha\le\pi$ ($\alpha=\pi$ gives full-aperture data) and we denote the radius of $\Gamma$ by $r$ in this section without the danger of confusion.  In the following, we first apply a data completion algorithm  to recover the full-aperture data in appropriate sense and then apply the modified sampling methods in the previous sections. The idea of the data completion algorithm is to represent the full-aperture data in the form of  Fourier series, and to relate the corresponding Fourier coefficients to the limited-aperture data via a prolate matrix.  This is in a similar fashion to the far-field case \cite{DLMZ2021} which considers two data completion algorithms: one based on Fourier series and the other one based on boundary integral equation. For completeness and to shed light on the near field case, we give the details on the data completion algorithm based on Fourier series in a modified way. We refer to \cite{DLMZ2021} for more related data completion algorithms.

\begin{notation}
In this section we are going to work with matrices and vectors with negative indexes for notational convenience. For $-N_1\le n\le N_2$, we denote by
$$
b:=\Big(b_n\Big), \quad -N_1\le n\le N_2
$$
as a $N_1+N_2+1$ dimensional vector.

For $-M_1\le m\le M_2$ and $-N_1\le n\le N_2$, we denote by
$$
A:=\Big(A_{mn}\Big), \quad -M_1\le m\le M_2, -N_1\le n\le N_2
$$
as a $(M_1+M_2+1) \times (N_1+N_2+1)$ dimensional matrix.

\end{notation}

\vspace{1\baselineskip}

\noindent\textit{Limited-aperture data using Fourier basis}: To begin with,  we introduce polar coordinates $x=r \cos \theta_x$  with $\theta_x \in (-\pi,\pi)$. At a fixed $y \in \Gamma_0$, for the limited-aperture backscattering far field measurements $u^{s}(x;y)$ with
$\theta_x \in [-\alpha, \alpha]$, we introduce the the infinite dimensional vector $C^\alpha_\infty$ with $p$-th entry given by $c^\alpha_{p}$
\begin{equation} \label{def C alpha}
c_{p}^{\alpha}:=\int_{-\alpha}^\alpha  u^{s}(x;y) \overline{\phi_p(\theta_x)  } \ind s_{\theta_x} ,\qquad p = 0, \pm 1,\cdots,
\end{equation}
here the Fourier basis is given by $\phi_p(\theta) = \frac{1}{\sqrt{2\pi}}e^{ip\theta}, \quad p=0, \pm1, \pm2,\cdots$.

\vspace{1\baselineskip}

\noindent\textit{Full-aperture data using Fourier basis}: At a fixed $y \in \Gamma_0$, for the full-aperture backscattering far field measurements $u^s(x,y)$ where $x \in \Gamma$, we introduce  the infinite dimensional vector $C_\infty$ with $p$-th entry given by $c_{p}$
\begin{equation} \label{def C}
c_{p} :=\int_{-\pi}^{\pi}  u^{s}(x;y)\overline{\phi_p(\theta_x)  }\ind s_{\theta_x} , \qquad p = 0, \pm 1,\cdots.
\end{equation}
The full-aperture backscattering far field measurements in the  Fourier basis correspond to the infinite dimensional matrix $C_\infty$.

Furthermore, given the knowledge of $C_\infty$, we can write down $u^{s}$ at a fixed $y \in \Gamma_0$ in  Fourier series as
\begin{equation}  \label{C to far field infinite}
u^s(x;y) = \sum_{m=-\infty}^{\infty}  c_{m}\phi_m(\theta_x), \quad \theta_x \in (-\pi,\pi),
\end{equation}
and approximate $u^s$ using a truncated Fourier series as
\begin{equation} \label{C to far field finite}
u^s(x;y) \approx \sum_{m=-J}^{J} c_{m}\phi_m(\theta_x)  , \quad \theta_x   \in (-\pi,\pi),
\end{equation}
for $J$ large enough so that the approximation error is sufficiently small in the $L^2$ sense.

\vspace{1\baselineskip}

\noindent\textit{Relation between limited-aperture data and full-aperture data}:
We first derive a relation between the limited-aperture data and full-aperture data as follows.
\begin{lemma} \label{thm far field case partial to full infinite}
Let $C^\alpha_\infty$ and $C_\infty$ be given by \eqref{def C alpha} and \eqref{def C} respectively. It holds that
\begin{equation} \label{far field case partial to full infinite}
C^{\alpha}_\infty = \mathbb{P}C_\infty,
\end{equation}
where the infinite dimensional matrix $\mathbb{P}$ is the prolate matrix (with dimension infinity)  given by \eqref{def Prolate Matrix} whose $mn$-th entry is given by
\begin{equation} \label{def Prolate Matrix}
\mathbbm{p}_{mn}
:= \int_{-\alpha}^\alpha \phi_m(\theta) \overline{\phi_n(\theta)} \ind \theta
= \frac{1}{2\pi}\int_{-\alpha}^\alpha e^{i(m-n)\theta} \ind \theta
=
\bigg\{
\begin{array}{cc}
\frac{\alpha}{\pi},  &  m=n    \\
\frac{\sin((m-n)\alpha)}{\pi (m-n)}, &     m\not=n
\end{array}
.
\end{equation}

\end{lemma}

\begin{proof}
Assume that the full-aperture measurements are given, then there is the Fourier series expansion \eqref{C to far field infinite}. Plugging this expression into the definition $\eqref{def C alpha}$  yields
$$
c^\alpha_{p} = \mathbbm{p}_{mp} c_{m}.
$$
Note that $\mathbb{P}$ is symmetric, this proves \eqref{far field case partial to full infinite} and completes the proof.
\end{proof}

%
%
%
%

\vspace{1\baselineskip}

\noindent\textit{Finite dimensional case}:
In practice, the measurements are discrete data. This motivates us to consider a finite dimensional space consisting of   $\phi_m(\theta) = \frac{1}{\sqrt{2\pi}}e^{im\theta}, \, m=0, \pm1, \cdots,\pm J$ for a sufficiently large $J$.  

The following theorem in the finite dimensional case follows immediately from Lemma \ref{thm far field case partial to full infinite}.
\begin{theorem} \label{thm far field case partial to full finite}
Let $C^\alpha:=\Big(c^\alpha_{p}\Big)_{-J \le p\le J}$ and $C:=\Big(c_{p}\Big)_{-J \le p\le J}$ with $c^\alpha_{p}$ and $c_{p}$ given by \eqref{def C alpha} and \eqref{def C} respectively. It holds that
\begin{equation} \label{far field case partial to full finite}
C^{\alpha} = \mathbb{P} C,
\end{equation}
where $\mathbb{P} $ is the $(2J+1)\times (2J+1)$ prolate matrix given by \eqref{def Prolate Matrix}.
\end{theorem}

\vspace{1\baselineskip}

\noindent\textit{From limited-aperture data to full-aperture data}:
Now it is clear that the limited-aperture data is related to the full-aperture data via  \eqref{far field case partial to full finite}. Our goal is then to find $C$ or its approximation from $C^\alpha $ via \eqref{far field case partial to full finite} at each fixed $y \in \Gamma_0$. From the properties of the prolate matrix in \cite[Lemma 3.2]{DLMZ2021}, we have that the eigenvalues of $\mathbb{P}$ are all positive, but they are clustered near $1$ and $0$, and hence the matrix $\mathbb{P}$  is severely ill-conditioned. In fact the eigenvalues decay exponentially to $0$ when $J$ becomes large. We refer to \cite{Slepian78,Varah1993} for more details on the prolate matrix. Therefore we can only hope to invert $\mathbb{P}$ using regularization techniques in order to find $C$ from $C^\alpha$. We shall apply
the following Regularization \eqref{data completion Tikhonov 1}
to find approximate inverse of $\mathbb{P}$.

\textbf{Regularization}: Similar to \cite{DLMZ2021}, our choice is to consider the regularization such that
\begin{equation} \label{data completion Tikhonov 1}
\mathbb{P}^\dagger = \mathbb{U} \left(\frac{1}{ \sigma_j + \epsilon}\right) \mathbb{U}^*,
\end{equation}
where $\epsilon>0$ is a regularization parameter, and
$$
\mathbb{P} = \mathbb{U} \Sigma \mathbb{U}^*, \qquad \Sigma=\mbox{diag}(\sigma_{-J},\sigma_{-J+1}, \cdots, \sigma_{J}).
$$
In this case, we take $\mathbb{P}^\dagger$ to approximate $\mathbb{P}^{-1}$. From the point of view of Slepian's spheroidal wave functions \cite{Slepian78,Varah1993}, this method attempts to use some information of the spheroidal wave functions with ``small energy'' on the interval $[-\alpha,\alpha]$.
Now we are ready to summarize the imaging algorithm with limited-aperture data.

\medskip

\noindent\textbf{Imaging Algorithm.}
\
\begin{itemize}
\item Data completion for $\{u^s(x;y): \, x\in \Gamma,\, y\in\Gamma_0\}$:  At each fixed $y\in\Gamma_0$, recover the full-aperture data $\{u^s(x;y): \, x\in \Gamma,\, y\in\Gamma_0\}$ approximately  using the limited-aperture data $\{u^s(x;y): \, x\in \Gamma_0, \,y\in \Gamma_0\}$.
\begin{itemize}
\item Step I: Compute $C^{\alpha}=\Big(c_{p}^{\alpha}\Big)$ from the measurements $\{u^s(x;y): \, x\in \Gamma_0\}$ by \eqref{def C alpha}.
\item Step II: Approximate $C$ by $ \mathbb{P}^{\dagger} C^\alpha $, where $\mathbb{P}^\dagger$ is the approximate inverse of $\mathbb{P}$ using
Regularization \eqref{data completion Tikhonov 1}.
\item Step III: Recover the full-aperture data by \eqref{C to far field finite}.
\end{itemize}
\item At each fixed $x\in\Gamma$, repeat the above Data Completion to recover the full-aperture data $\{u^s(x;y): \, y\in \Gamma,\, x\in\Gamma\}$ approximately  using the recoved data $\{u^s(x;y): \, y\in \Gamma_0, \,x\in \Gamma\}$ in the above step (note the reciprocity relation $u^s(x;y)=u^s(y;x)$ \cite{CK}).
\item Modified sampling method: Reconstruct the object by the imaging function  \eqref{imaging function obstacle} for the obstacle case and \eqref{imaging function cavity} for the cavity case.
\end{itemize}

\section{Numerical Examples} \label{section numerical examples}

In this section, we present some numerical examples to illustrate the performance of the modified sampling methods proposed in the previous sections.
The numerical examples are divided into three groups.
\begin{itemize}
  \item Reconstructions of obstacles using the imaging function \eqref{imaging function obstacle};
  \item Reconstructions of cavities using the imaging function  \eqref{imaging function cavity};
  \item Reconstructions of obstacles with limited-aperture data \eqref{def limited data}.
\end{itemize}
The boundaries of the scatterers used in our numerical experiments are parameterized as follows:
\be
\label{circle}&\mbox{\rm Circle:}&\quad x(t)\ =(a,b)+\, r(\cos t, \sin t),\quad 0\leq t\leq2\pi,\\
\label{ellipse}&\mbox{\rm Ellipse:}&\quad x(t)\ =(a,b)+\, (2\cos t, 3\sin t),\quad 0\leq t\leq2\pi,\\
\label{roundsquare}&\mbox{\rm Round Square:}&\quad x(t)\ =(a,b)+(1.5\cos^{3}t +1.5\cos t, 1.5\sin^{3}t +1.5\sin t),\quad 0\leq t\leq2\pi,\\
\label{peanut}&\mbox{\rm Peanut:}&\quad x(t)\ =(a,b)+\, 1.5\sqrt{3\cos^2 t+1}(\cos t, \sin t),\quad 0\leq t\leq2\pi,\\
\label{kite}&\mbox{\rm Kite:}&\quad x(t)\ =(a,b)+\, (1.1\cos t+0.625\cos 2t-0.625, 1.5\sin t),\quad 0\leq t\leq2\pi,
\en
with $(a,b)$ be the location of the scatterer which will be specified in different examples.

In our simulations, the boundary integral equation method is used to compute the
scattered fields $u^{s}(x;y)$ for $L$ equidistantly distributed measurement points  and $L$ equidistantly distributed source points over the measurement surface. These data are then stored in the matrices $\mathbb{N} \in \C^{L \times L}$.
We further perturb $\mathbb{N}$ by random noise using
\ben
\mathbb{N}^{\delta}\ =\ \mathbb{N} +\delta\|\mathbb{N}\|\frac{R_1+R_2 i}{\|R_1+R_2 i\|},
\enn
where $R_1$ and $R_2$ are two $L \times L$ matrices containing pseudo-random values
drawn from a normal distribution with mean zero and standard deviation one. The
value of $\delta$ used in our code is $\delta:=\|\mathbb{N}^{\delta} -\mathbb{N}\|/\|\mathbb{N}\|$ which represents the relative error.

\subsection{Obstacle}
We first  illustrate the performance of the   imaging function $I_{obstacles}$ for obstacles.
In the simulations, we use a grid $\mathcal{G}$ of $301\times 301$ equally spaced sampling points on the rectangle $[-5,5]\times[-5,5]$.
We use $L=128$ equidistant sensors $y_i,i=1,2,\cdots,128$  on the circle $\partial B_5:=\{y\in\R^2:\, |y|=5\}$.
We take $k=10$. Therefore the wavelength $\lambda=2\pi/k\approx 0.618$ and consequently the measurement surface $\pa B_5$ is just two or three wavelengths away from the obstacles.
As suggested in the arguments after Theorem \ref{obstacle Iz theorem}, we take $M=32$ in the definition $\varphi_z$ of \eqref{obstacle varphi def}.
For each sampling point $z \in \mathcal{G}$, we define the indicator function
\be\label{obstacleindicator}
W_{obstacle}(z)\ :=|\Phi_z^{T}\mathbb{N}^{\delta}\Phi_z|,
\en
where $\Phi_z=(\varphi_z(y_1), \varphi_z(y_2),\cdots, \varphi_z(y_{128}))^\top \in \C^{128}$.
Clearly, the indicator function is independent of any a priori information of the unknown obstacles.

We report five examples. In the first three examples, the underlying obstacles are round square, peanut and kite located at $(0,0)$, respectively. In the 
fourth example the obstacle $\Omega$ is given by the union of two disjoint components $\Omega:=\Omega_1\cup\Omega_2$, where $\Omega_1$ is a disk with radius $1$ and center $(-2.5, 0)$ while $\Omega_2$ is a kite located at $(2,0)$. For the fifth example the obstacle $\Omega$ is also given by the union of two disjoint components $\Omega:=\Omega_1\cup\Omega_2$, where $\Omega_1$ is a disk with radius $1$ and center $(0, 2.5)$ while $\Omega_2$ is a peanut located at $(-1.2,0)$.
These five exact obstacles are shown in the first column of Figure \ref{obstacles}.

Reconstructions are shown in the second and third column of Figure \ref{obstacles}. Obviously, the indicator function $W_{obstacle}$ peaks on the boundaries of the obstacles and is capable to reconstruct  the locations and shapes. As shown in the last two examples, the different components are also well reconstructed. 
As shown in the third column of Figure \ref{obstacles}, the reconstructions with $10\%$ noise are almost the same as the those without noises. This further implies that our modified sampling methods are quite stable with respect to noises. These reconstructions are comparable to those using similar imaging function with far-field measurements \cite{LiuIP17}.

\begin{figure}[htbp]
  \centering
\subfigure[]{
    \includegraphics[width=4in]{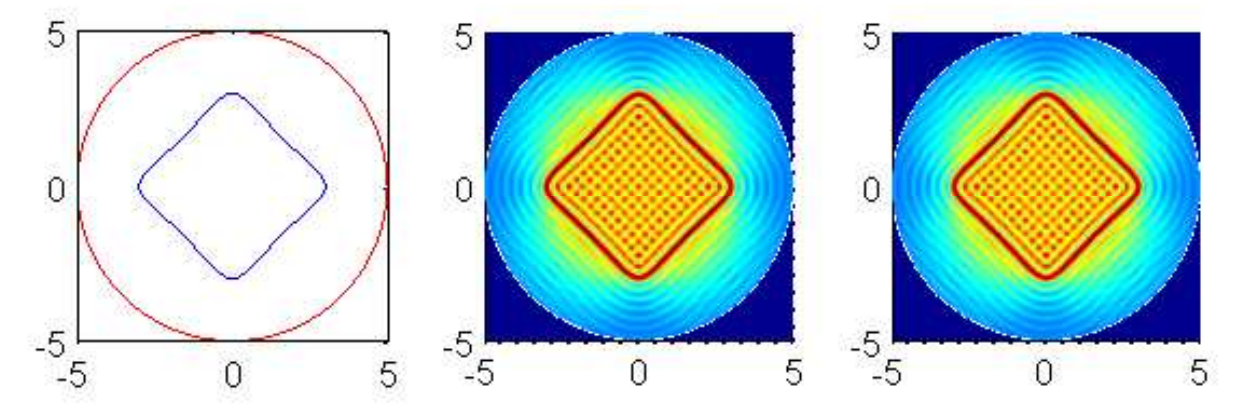}}\\
\subfigure[]{
    \includegraphics[width=4in]{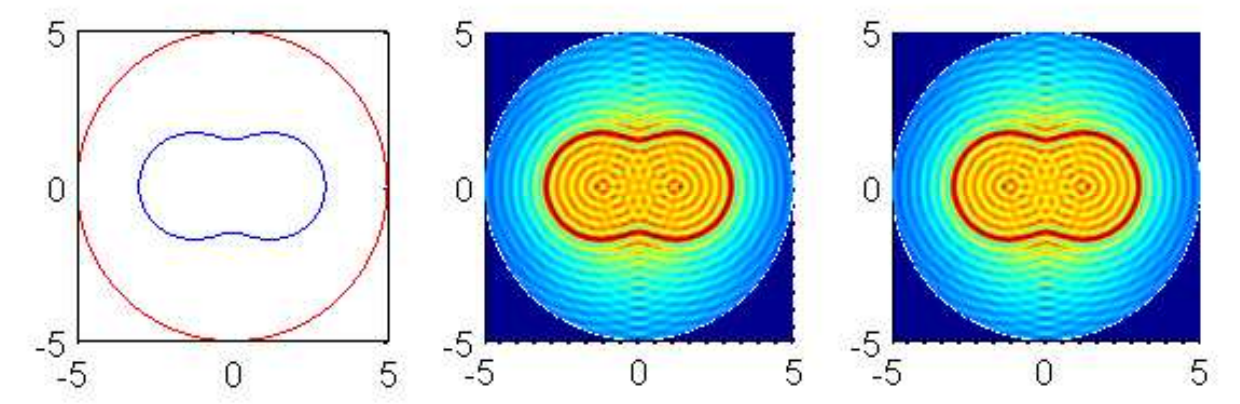}}\\
  \subfigure[]{
    \includegraphics[width=4in]{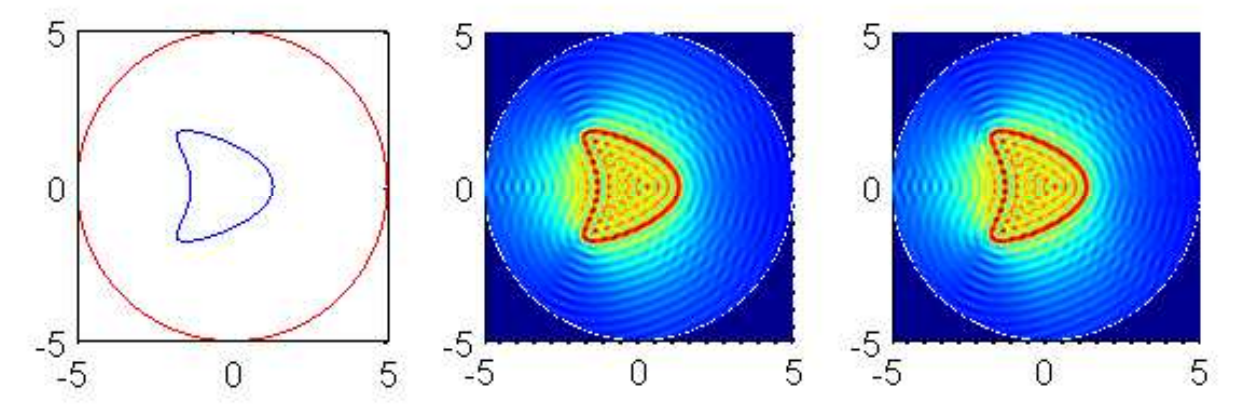}}\\
      \subfigure[]{
    \includegraphics[width=4in]{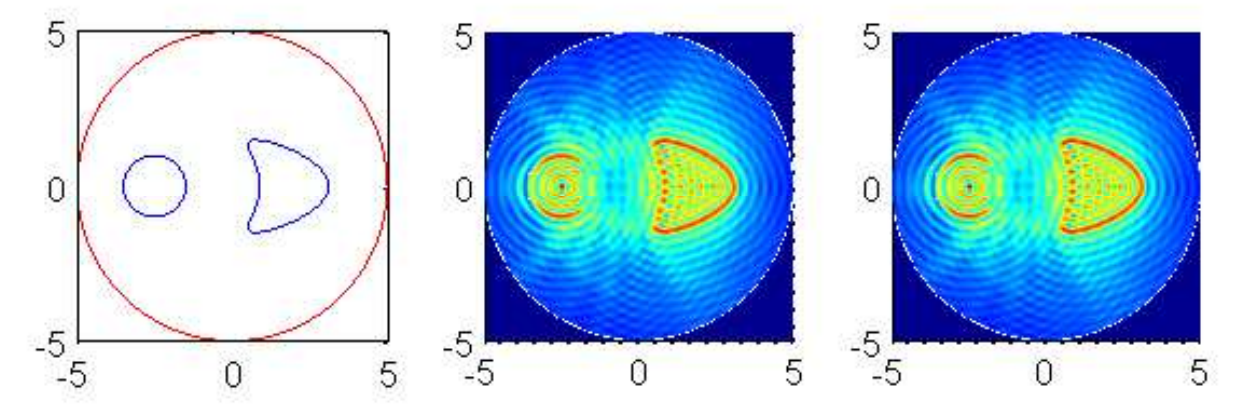}}\\
    \subfigure[]{
    \includegraphics[width=4in]{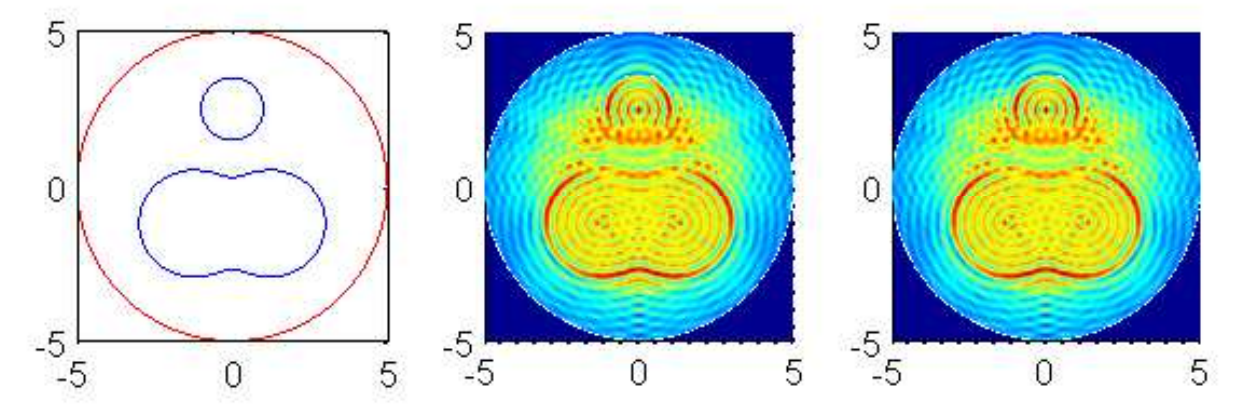}}
\caption{{\bf Reconstructions of obstacles.} The first column is the exact obstacles. The second column shows the reconstructions without noises, while the third column shows the reconstructions with $10\%$ noise.}
\label{obstacles}
\end{figure}

\subsection{Cavity}
This subsection is denoted to  illustrating the performance of the   imaging function $I_{cavity}$ for cavities. 
In the simulations, we used a grid $\mathcal{G}$ of $81\times 81$ equally spaced sampling points on the rectangle $[-4,4]\times[-4,4]$.
As suggested in Remark \ref{Mchoise}, we take $\mathfrak{M}=3$ in the definition $\psi_z$ of \eqref{cavity psi def}.
For each sampling point $z \in \mathcal{G}$, we define the indicator function
\ben
W_{cavity}(z)\ :=|\Psi_z^{\ast}\mathbb{N}^{\delta}\Psi_z|,
\enn
where $\Psi_z=(\psi_z(y_1), \psi_z(y_2),\cdots, \psi_z(y_{128}))^\top \in \C^{L}$ and the superscript $*$ denotes the conjugate transpose.

The cavities we considered are round square, disk, ellipse, peanut and kite with center $(0,0)$. The radius of the disk is $2$.
We refer to the first column of Figure \ref{cavities} for the exact cavities. 
In the cavity case, general reconstructions using low frequency waves seem to perform better than those using high frequency waves \cite{ZM2021}, we therefore  take the wave number $k=0.2$. 

In Figure \ref{cavities}, we show the reconstructions of various cavities in the second column (without noises) and third column (with $1\%$ noise). We find that the indicator function $W_{cavity}$ starts to peak as the sampling point approaches the boundary from the interior of the cavities and this allows us to reconstruct  the shapes.  As shown in the third column of Figure \ref{cavities}, the reconstructions are not that robust to noises compared to the obstacle case.  

As discussed in  Remark \ref{Mchoise}, measurement circle with smaller radius may lead to numerical instability for noisy measurements, we illustrate this in Figures \ref{cavities}-\ref{cavities0005} by considering three measurement circles with radius $r=1$, $r=0.5$ and $r=0.005$, respectively.  Without noise, the proposed imaging functional $W_{cavity}$ gives quite good reconstructions. Surprisingly, as shown in Figure \ref{cavities0005}, even the measurement surface is very small and there are only $8$ sensors, the cavities can be well reconstructed. However, the reconstructions are more sensitive to noises for measurement circle with smaller radius.

\begin{figure}[htbp]
  \centering
  \subfigure[]{
    \includegraphics[width=4in]{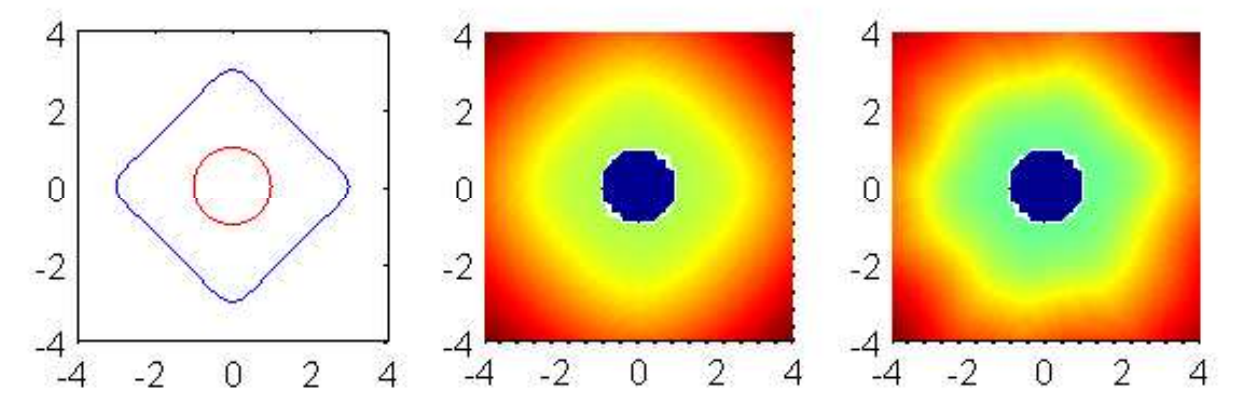}}\\
\subfigure[]{
    \includegraphics[width=4in]{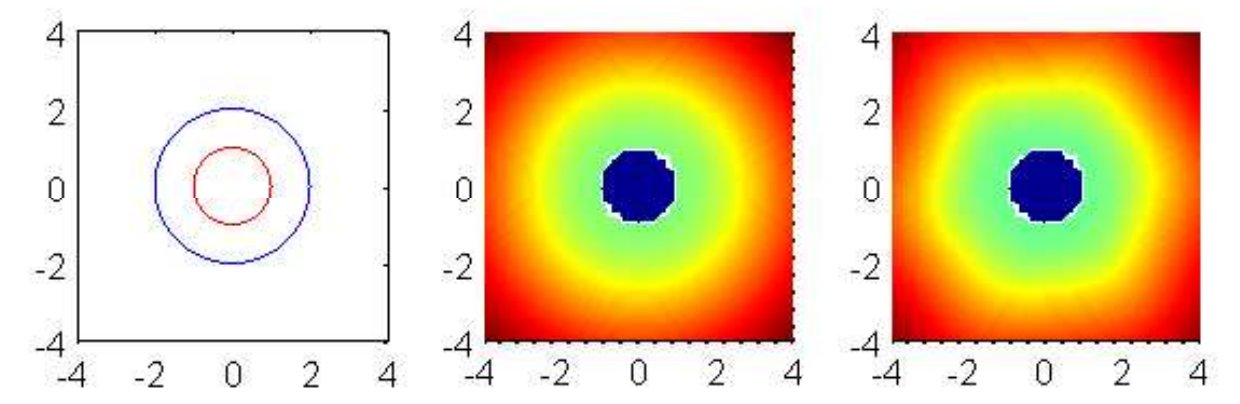}}\\
  \subfigure[]{
    \includegraphics[width=4in]{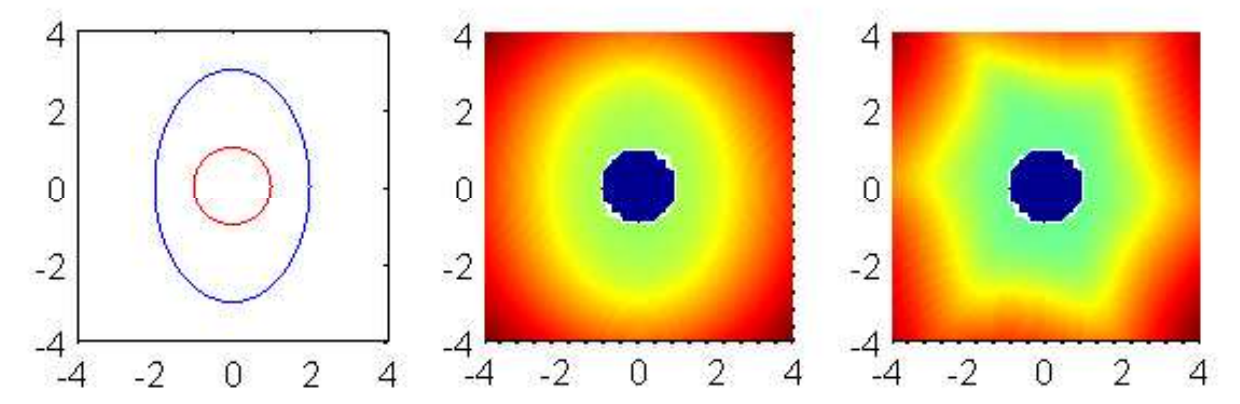}}\\
\subfigure[]{
    \includegraphics[width=4in]{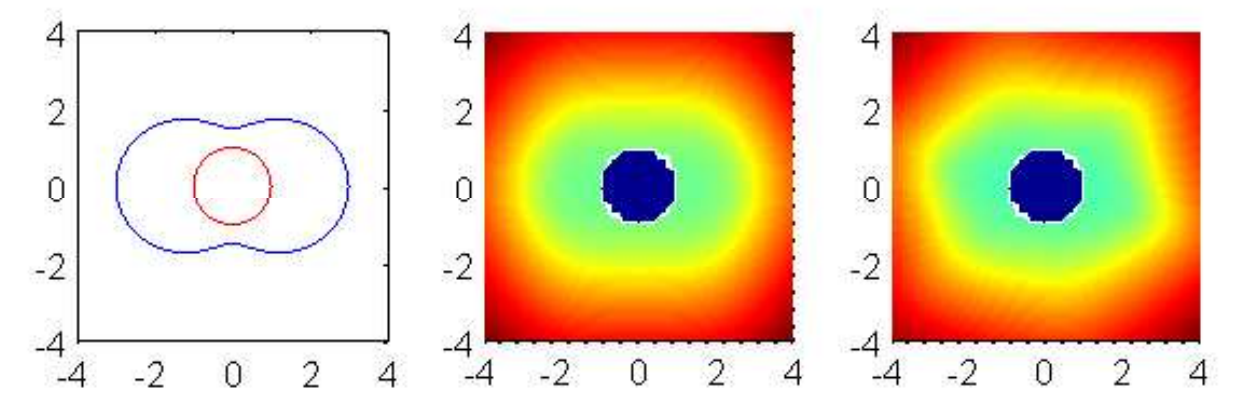}}\\
  \subfigure[]{
    \includegraphics[width=4in]{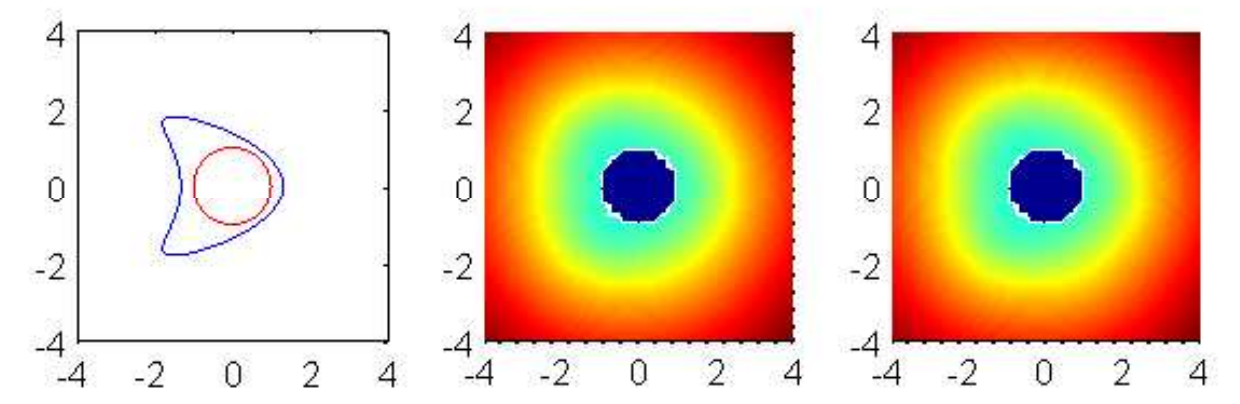}}
\caption{{\bf Reconstructions of cavities.} There are $L=32$ sensors on the measurement circle with radius $1$. The first column is the exact cavities. The second column shows the reconstructions without noises, while the third column shows the reconstructions with $1\%$ noise.}
\label{cavities}
\end{figure}

\begin{figure}[htbp]
  \centering
  \subfigure[]{
    \includegraphics[width=4in]{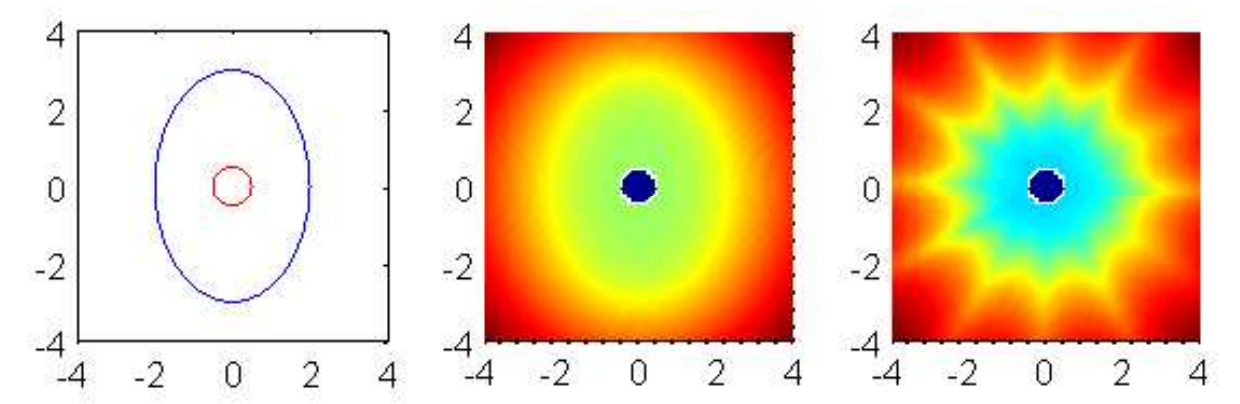}}\\
\subfigure[]{
    \includegraphics[width=4in]{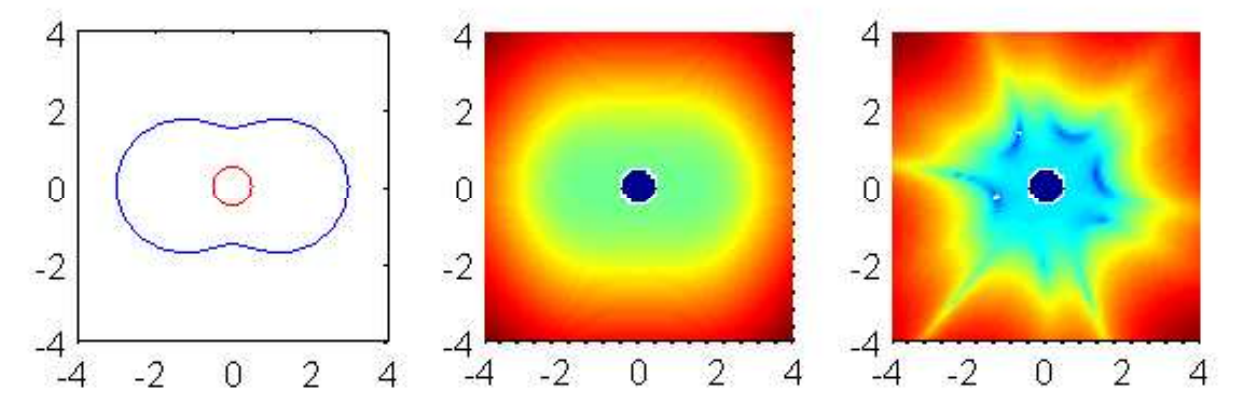}}
\caption{{\bf Reconstructions of cavities.} There are $L=32$ sensors on the measurement circle with radius $0.5$.  The second column shows the reconstructions without noises, while the third column shows the reconstructions with $1\%$ noise.}
\label{cavities05}
\end{figure}

\begin{figure}[htbp]
  \centering

  \subfigure[]{
    \includegraphics[width=4in]{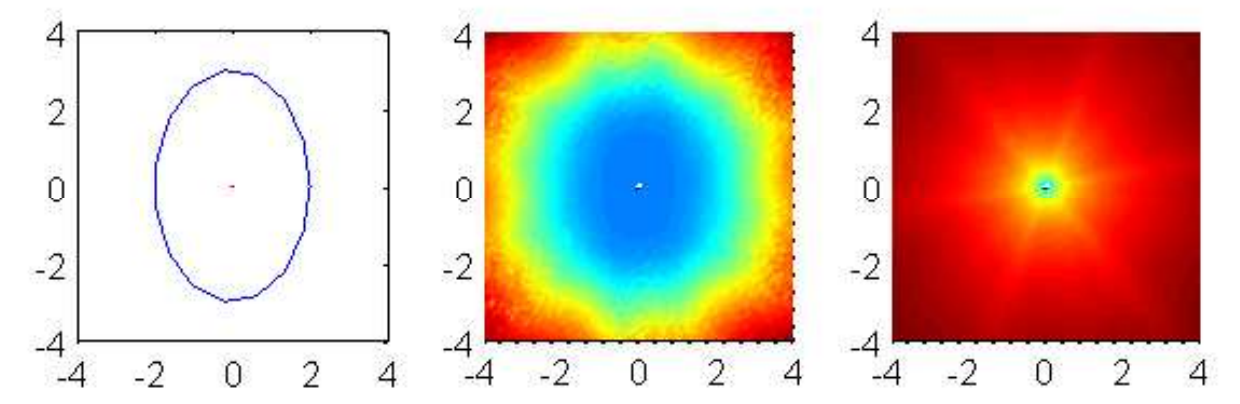}}\\
\subfigure[]{
    \includegraphics[width=4in]{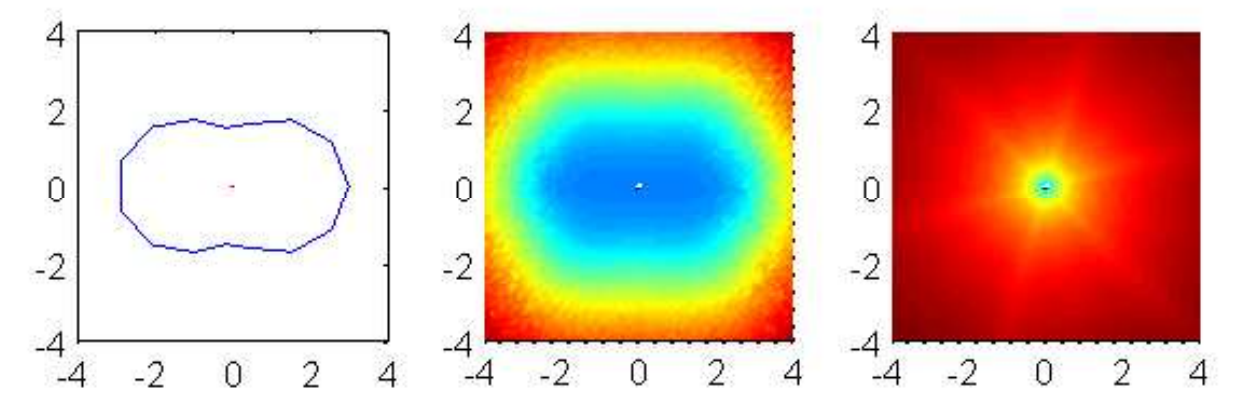}}
\caption{{\bf Reconstructions of cavities.} There are $L=8$ sensors on the measurement circle with radius $0.005$. The second column shows the reconstructions without noises, while the third column shows the reconstructions with $1\%$ noise.}
\label{cavities0005}
\end{figure}

\subsection{Limited-aperture}
In this subsection, we illustrate the performance of our modified sampling methods using limited-aperture data. We consider $64$ sensors (of the total $128$ sensors)  equidistantly located on the upper half circle $\pa B_{+}:=\{(x_1,x_2)\in \pa B_5: \, x_2>0\}$.
These data are then stored in the matrices $\mathbb{N}_{partial} \in \C^{64 \times 64}$.
We consider the following two strategies for the limited-aperture problems.
\begin{itemize}
  \item Reconstruct the obstacles using the limited-aperture data directly with the indicator function
  \ben
W_{limit}(z)\ :=|\widetilde{\Phi}_z^{T}\mathbb{N}_{partial}\widetilde{\Phi}_z|,
\enn
where $\widetilde{\Phi}_z=(\varphi_z(y_1), \varphi_z(y_2),\cdots, \varphi_z(y_{64}))^\top \in \C^{64}$.
  \item Recover the full aperture data using the data completion algorithm in Section \ref{section limited-aperture}, where the Regularization \eqref{data completion Tikhonov 1} is used with parameter $\epsilon = 10^{-3}$. With the recovered full aperture data, we then apply the modified sampling methods to reconstruct the obstacles with the indicator function \eqref{obstacleindicator}.
\end{itemize}

The first two columns of Figure \ref{limitaperture} show the reconstructions using these two strategies, respectively. For comparison, we show in the third column of   Figure \ref{limitaperture} the results using full aperture data. $10\%$ noise is added in all of these reconstructions.
It is observed that the upper half parts of the obstacles can be well reconstructed and the resolution indeed can be improved with the help of the data completion algorithm.

\begin{figure}[htbp]
  \centering
\subfigure[]{
    \includegraphics[width=4in]{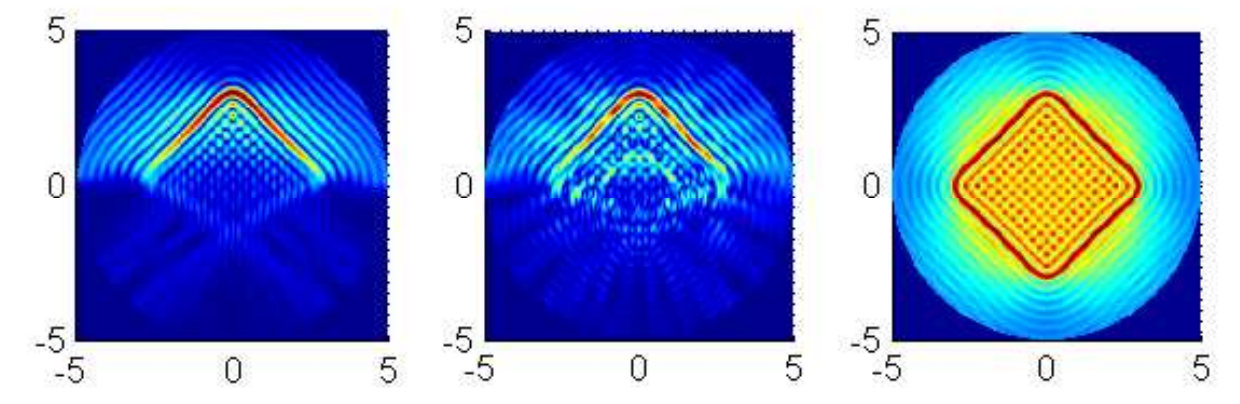}}\\
  \subfigure[]{
    \includegraphics[width=4in]{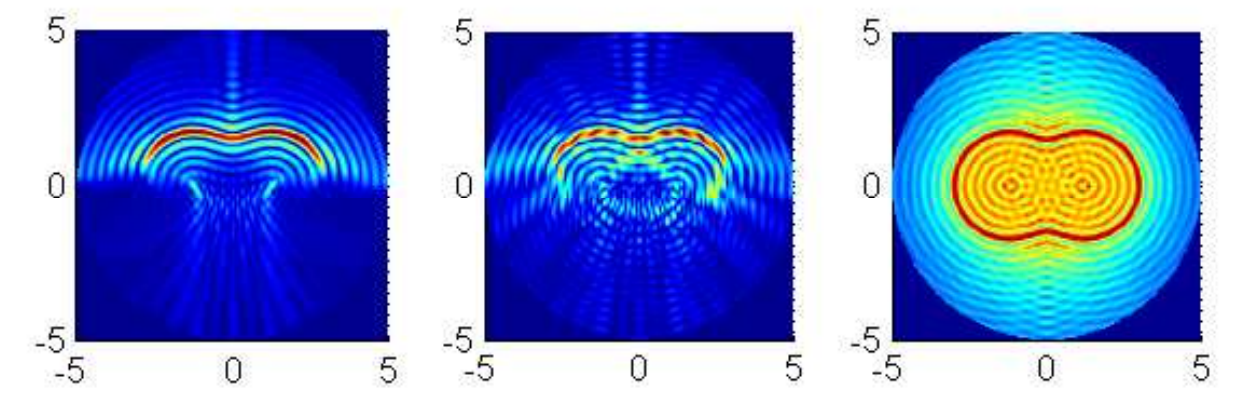}}\\
  \subfigure[]{
    \includegraphics[width=4in]{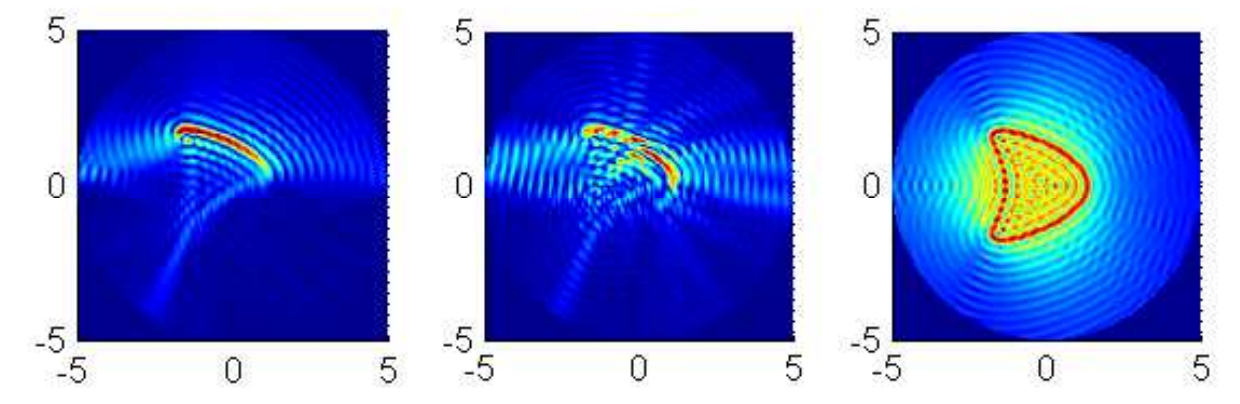}}\\
      \subfigure[]{
    \includegraphics[width=4in]{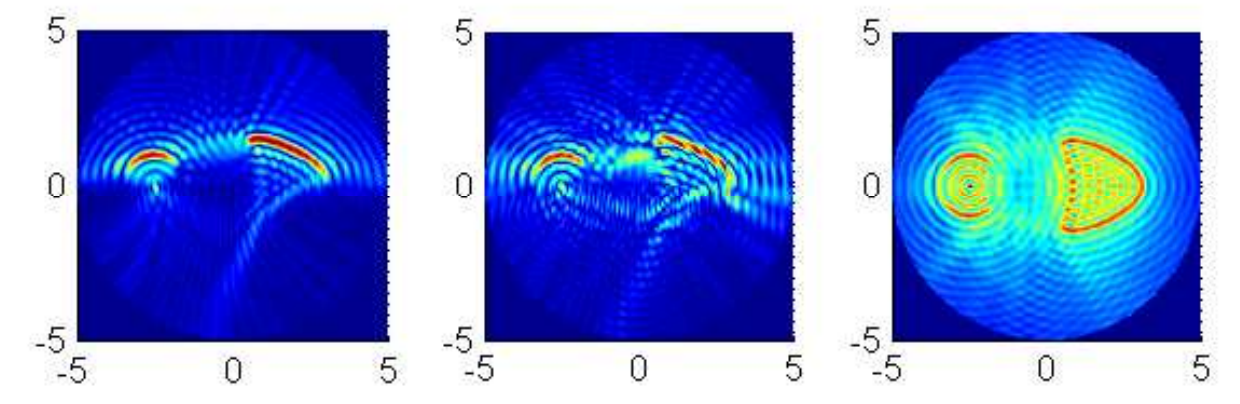}}\\
    \subfigure[]{
    \includegraphics[width=4in]{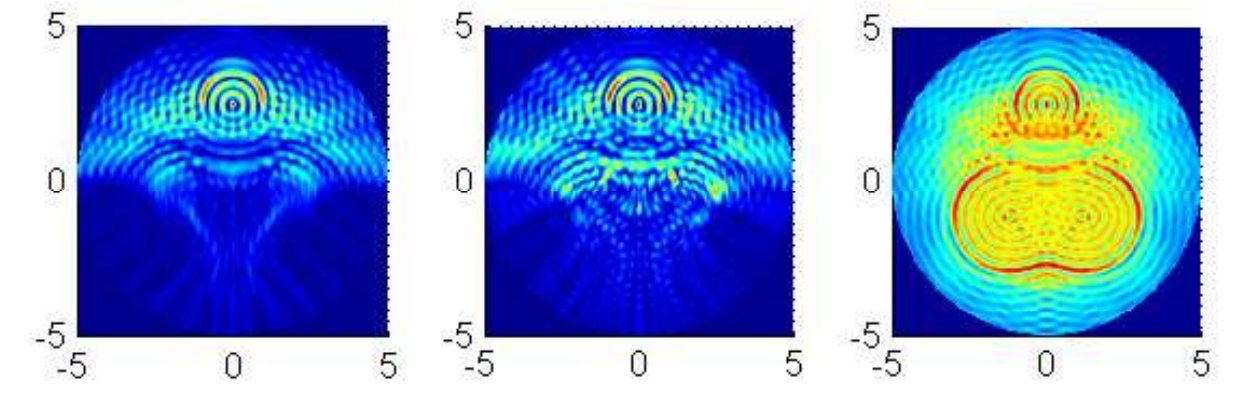}}
\caption{ {\bf Reconstructions using limited-aperture measurements.} The first two columns show the reconstructions with measurements on  $\pa B_+$.
The first column is the reconstructions using limited-aperture data directly, while the second column shows the reconstructions combined with the data completion algorithm. The third column shows the reconstructions with full aperture data.}
\label{limitaperture}
\end{figure}

\section*{Acknowledgments}
The research of X. Liu is supported by the NNSF of China grant 11971471 and the Youth Innovation Promotion Association, CAS. The research of B. Zhang is partially supported by the NNSF of China grant 91630309.


\end{document}